\documentclass[11 pt]{amsart}
\usepackage[T1]{fontenc}\usepackage{amssymb,array, mathtools, amsthm, booktabs, float}

\usepackage[backend=biber,style=alphabetic]{biblatex}
\renewbibmacro*{in:}{%
	\ifentrytype{article}
	{}
	{\printtext{\bibstring{in}\intitlepunct}}}
\DeclareFieldFormat[article]{title}{\mkbibemph{#1}}
\DeclareFieldFormat[article]{journaltitle}{#1}
\DeclareFieldFormat[article]{volume}{\mkbibbold{#1}}
\DeclareFieldFormat[article]{number}{nº~#1}
\renewbibmacro*{volume+number+eid}{%
	\printfield{volume}%
	\setunit*{\addcomma\space}
	\printfield{number}%
	\setunit{\addcomma\space}%
	\printfield{eid}}
\usepackage{hyperref}
\hypersetup{hidelinks}

\usepackage{caption}
\newtheorem{theorem}{Theorem}[section]

\newtheorem{lemma}[theorem]{Lemma}
\theoremstyle{definition}
\newtheorem{definition}[theorem]{Definition}
\theoremstyle{remark}
\newtheorem{remark}[theorem]{Remark}
\theoremstyle{plain}
\newtheorem{proposition}[theorem]{Proposition}
\newtheorem{thmAlpha}{Theorem}

\newtheorem{propAlpha}[thmAlpha]{Proposition}

\DeclareMathOperator{\Lie}{Lie}
\DeclareMathOperator{\ric}{ric}
\DeclareMathOperator{\Ric}{Ric}
\DeclareMathOperator{\Der}{Der}
\DeclareMathOperator{\Id}{Id}
\DeclareMathOperator{\diag}{diag}
\DeclareMathOperator{\ad}{ad}
\DeclareMathOperator{\Tr}{Tr}
\DeclareMathOperator{\vol}{vol}
\DeclareMathOperator{\grad}{grad}
\DeclareMathOperator{\End}{End}

\newcommand{\brc}{\Ric^B_{g,H}}

\numberwithin{equation}{section}
\allowdisplaybreaks

\title{SOLVABLE ALGEBRAIC GENERALIZED RICCI SOLITONS IN LOW DIMENSIONS}
\author{Lorenzo Mucciante}
\address{Dipartimento di Matematica e Informatica ``U.~Dini'' \\ Universit\`a degli Studi di Firenze\\
	 Viale Morgagni 67/A\\ 50134 Firenze\\ Italy\\ }
\email{lorenzo.mucciante@unifi.it}
\subjclass[2020]{53E20, 53C30, 22E25}
\keywords{generalized Ricci flow, generalized Ricci soliton, solvable Lie algebra, homogeneous Riemannian space.}

\begin{document}
	 \begin{abstract}
	We prove new structural results for algebraic generalized Ricci solitons on solvable Lie algebras and we discuss the existence of examples in low dimensions. We establish that all non-trivial solvable algebraic generalized Ricci solitons must be expanding and give necessary conditions for the existence of such solitons on one-dimensional extensions of abelian Lie algebras. Finally, we classify algebraic generalized Ricci solitons on three-dimensional Lie algebras, and algebraic generalized Ricci solitons on four-dimensional solvable unimodular Lie algebras, up to isometry and scaling.
	 \end{abstract}
	\maketitle
	\pagestyle{plain}
	
	\section{Introduction}
Let $M$ be an oriented smooth manifold endowed with a Riemannian metric $g$ and with a closed $3$-form $H$. The pair $(g,H)$ is called a \emph{generalized Riemannian metric} on $M$. It is said to be a \emph{generalized Ricci soliton} (GRS for short) if there are a vector field $X \in \Gamma(TM)$ and a constant $\lambda \in \mathbb{R}$ such that
	\begin{equation*}
		\begin{cases}
			\Ric_g=\lambda g-\frac{1}{2}\mathcal{L}_Xg+\frac{1}{4}H_g^2\\
			\Delta H = 2\lambda H - \mathcal{L}_XH
		\end{cases}
	\end{equation*}
	where $\Delta$ is the Hodge Laplacian induced by the metric and the orientation, while $H_g^2$ is the symmetric (0,2)-tensor defined by
	\begin{equation*}
		H_g^2(Y,Z):=g(\imath_Y H,\imath_Z H).
	\end{equation*}
	
	We will say that a soliton is \emph{expanding} if $\lambda<0$, \emph{shrinking} if $\lambda>0$ and \emph{steady} if $\lambda=0$. Moreover, if $X=\grad(f)$ for some smooth function $f$ we will say that the soliton is of \emph{gradient type}. 
	
	Generalized Ricci solitons are closely related to the \emph{generalized Ricci flow} (GRF for short), a geometric flow that evolves generalized metrics according to
	\begin{equation*}
		\begin{cases}
			\partial_tg_t=-2\Ric^B_{g_t,H_t}\\
			\partial_tH_t=-\Delta_{g_t}H_t
		\end{cases}
	\end{equation*}
	where we set $\Ric^B_{g_t,H_t}=\Ric_{g_t}-\frac{1}{4}H^2_{g_t}$.
	More precisely, a GRF solution starting at a generalized soliton is self-similar, meaning that it evolves only by diffeomorphisms and simultaneous scalings of the pair $( g,H)$. This flow first appeared in the physics literature in \cite{CALLAN1985593} as renormalization group flow. Later, in \cite{OLIYNYK2006441} the authors show that on compact manifolds it can be seen as gradient flow for a certain Schr\"odinger operator. 
	 
	Similarly to the classical case, generalized Ricci solitons are expected as singularity models and asymptotic limits for the GRF. For instance, in \cite{fusi2024homogeneousgeneralizedricciflow} it is proved that on nilmanifolds the flow converges asymptotically to a GRS in a topology that generalizes the usual Cheeger--Gromov, provided that the initial form is harmonic.
	Special solutions to the GRF have been studied in recent works. Spaces with flat Bismut connection are the fixed points of the GRF and in \cite{Podest__2022_A} an example of Bismut-Ricci flat manifold that is non-flat is given (see also \cite{Podest__2022_B}). Also, the $SO(3)$-invariant, steady, gradient GRS analogue of the classical Bryant soliton has been built in \cite{Podest__2025}.
	
	On the other hand, there are many rigidity results on compact GRS that force the 3-form $H$ to vanish. For instance, by employing suitable energy functionals, \cite{STREETS2008900} and \cite{li2024shrinkingsolitonsgeneralizedricci} showed that this vanishing occurs for both expanding and shrinking gradient GRS. Furthermore, all compact steady generalized solitons are necessarily of gradient type, as proved in \cite{garcia2021generalized}. 
	
	An important class of examples is given by invariant generalized metrics on homogeneous Riemannian spaces. If $(M, g)$ is a $G$-homogeneous Riemannian space, we can give a presentation of $M$ as the coset space $M \cong G/K$. Here $G$ is a transitive Lie subgroup of $\operatorname{Iso}(M,g)$ and $K$ its isotropy subgroup at a point $o\in M$. For such spaces we can relate the GRS equation on $M$ to an algebraic equation on the Lie algebra $\mathfrak{g}$ of $G$. More precisely, we say that a metric Lie algebra $(\mathfrak{g}, g)$ endowed with a closed $3$-form $H\in \Lambda^3\mathfrak{g}^*$ is an \emph{algebraic generalized Ricci soliton} if there exist a symmetric derivation $D\in \Der(\mathfrak{g})$ and $\lambda \in \mathbb{R}$, such that
	\begin{equation*}
		\begin{cases}
			\ric-\frac{1}{4}g^{-1}H^2=\lambda\text{Id}+ D\\
			\Delta H=2\lambda H-\rho(D)H.
		\end{cases}
	\end{equation*}
	where 
	\begin{equation*}
		\rho(E)H=-H(E\cdot,\cdot,\cdot)-H(\cdot,E\cdot,\cdot)-H(\cdot,\cdot,E\cdot)
	\end{equation*}
	for every $E\in \End(\mathfrak{g})$, and $\ric$ is the Ricci operator on $\mathfrak{g}$.
	
	This notion was first introduced to characterize the solitons of the generalized bracket flow, a flow of Dorfman brackets on the Courant algebroid $\mathfrak{g}\oplus\mathfrak{g}^*$. For further details on this topic and on generalized homogeneous geometry we refer to the papers \cite{fusi2024homogeneousgeneralizedricciflow} (where the authors also provide a classification of generalized nilsolitons) and \cite{fusi2026homogeneousgeneralizedricciflows}. Other examples of generalized nilsolitons are presented in \cite{chen2026harmonicityexistencealgebraicgeneralized}. In particular, the authors construct a generalized nilsoliton with $H\neq0$ on a Lie algebra admitting no classical Ricci nilsolitons.
	
	 Classical algebraic solitons on solvmanifolds are well understood, see for instance \cite{Jablonski+2015+159+182}, \cite{Lafuente2014}, and \cite{Jablonski_2014}. They have been classified in low dimensions in \cite{Lauret+2011+1+21} and are either Ricci-flat or expanding. This behaviour persists in the generalized case on nilpotent Lie algebras, see \cite{fusi2024homogeneousgeneralizedricciflow}.
	
	In this paper we prove that this actually holds on solvable Lie algebras. The first result is the following
	\begin{thmAlpha}[see Theorem \ref{teoA} below]
		\label{A}
		Let $(\mathfrak{s}=\mathfrak{a}\oplus \mathfrak{n}, \langle \cdot, \cdot \rangle)$ be a solvable metric Lie algebra, where $\mathfrak{n}$ is the nilradical and $\mathfrak{a}$ its orthogonal complement. Assume that $(\langle \cdot, \cdot \rangle, H)$ is an algebraic GRS for $H \in \Lambda^3\mathfrak{s}^*$ with $H\neq 0$. Then it is necessarily expanding.
	\end{thmAlpha}
\noindent To prove this we will make use of techniques and formulas from real geometric invariant theory, in the framework developed in \cite{Lauret2006}, \cite{Lauret2010}, \cite{Lauret+2011+1+21}.
Using these tools it is also possible to prove structural theorems for GRS on another class of solvable unimodular Lie algebras. Namely, if $\mathfrak{n}$ is abelian and $\mathfrak{a}$ is one-dimensional, the 3-form $H$ lies in $(\mathfrak{a}^*\otimes\Lambda^2\mathfrak{n}^*) \oplus \Lambda^3 \mathfrak{n}^*$. This decomposition allows one to split the soliton equation for $H$ into equations along its components. Our second structural result reads as follows
\begin{propAlpha}[see Propositions \ref{teoB}-\ref{propLaplaciano} below]
	\label{propB}
	Let $(\mathfrak{s}=\mathfrak{a}\oplus \mathfrak{n}, \langle \cdot, \cdot \rangle,H)$ be a solvable unimodular GRS with respect to a derivation $D$, where $\mathfrak{n}$ is the abelian nilradical and the orthogonal complement $\mathfrak{a}=\mathbb{R}A$ is one-dimensional. Write $H=\alpha \wedge \omega+ \eta \in (\mathfrak{a}^*\otimes\Lambda^2\mathfrak{n}^*) \oplus \Lambda^3 \mathfrak{n}^* $, where $\alpha=A^\flat$. Then
	\begin{enumerate}
		\item The form $\eta$ is harmonic and
		\begin{equation*}
			\begin{cases}
				\rho(\ad_A)\rho(\ad_A^T)\omega=2\lambda \omega-\rho(D) \omega\\
				2\lambda\eta=\rho(D)\eta.
			\end{cases}
		\end{equation*}
		\item If $\rho(\ad_A)\omega=0$ (that is, $\omega$ is $\mathfrak{a}$-invariant), then $\ad_A^T \in \Der(\mathfrak{s})$.
	\end{enumerate}
\end{propAlpha}
    
    In Section \ref{sec_dim3} we give the classification of algebraic GRS in dimension 3. We leave nilpotent Lie algebras out of this treatment since they are treated in \cite{fusi2024homogeneousgeneralizedricciflow}. The only algebraic GRS that we find is $\mathfrak{su}(2)$ endowed with a Bismut-Ricci flat metric, see Theorem \ref{class_3d} below.
    
    In Section \ref{sec_dim4} we give a detailed treatment of algebraic GRS on  4-dimensional, non-nilpotent, solvable unimodular Lie algebras. In particular, Proposition \ref{propB} allows us to give a complete classification of these algebraic GRS, extending the results in \cite{fusi2024homogeneousgeneralizedricciflow} to the solvable case. The result is summarized in the following
    \begin{propAlpha}
    	For every $4$-dimensional, non-nilpotent, solvable unimodular Lie algebra, exactly one of the following holds:
    	\begin{enumerate}
    		\item There exists a unique GRS up to scaling with $H\neq0$;
    		\item There do not exist any GRS with $H\neq0$;
    	\end{enumerate}
    \end{propAlpha}
\noindent The explicit GRS are shown in the Propositions \ref{class1}-\ref{class2} below and summarized in Table \ref{table3}.

\vspace{0.5cm}
\noindent	\textbf{Acknowledgements:} This work is part of my PhD project. I am deeply thankful to my PhD advisors Fabio Podestà and Alberto Raffero for suggesting the problem, pointing out inaccuracies in the first draft of this work and for their valuable comments regarding the content and organization of this paper.
	
\noindent The author is supported by the University of Florence and by GNSAGA of INdAM.
	\section{Homogeneous generalized Ricci solitons}
	In this section we recall the derivation of the generalized Ricci soliton equations for homogeneous spaces.
	
	\begin{definition}
		\label{GRS}
		A \emph{generalized Ricci soliton} (shortly GRS) on a smooth manifold $M$ consists of the data $(g,H,X, \lambda)$ where $g$ is a Riemannian metric, $H$ is a closed 3-form,  $X\in \Gamma(TM)$ is a vector field and $\lambda \in \mathbb{R}$ such that
		\begin{equation*}
			\begin{cases}
				\text{Ric}=\lambda g-\frac{1}{2}\mathcal{L}_Xg+\frac{1}{4}H_g^2\\
				\Delta H = 2\lambda H - \mathcal{L}_XH.
			\end{cases}
		\end{equation*}
	\end{definition}
Here $H^2$ is the symmetric $(0,2)$-tensor defined by $H_g^2(Y,Z):=g(\imath_Y H,\imath_Z H)$ and $\Ric$ denotes the Ricci tensor of the Levi-Civita connection of $g$. We shall omit the subscript in $H_g$ when the metric is understood.

Recall that the Bismut connection is the unique metric connection with totally skew-symmetric torsion. It is given by
\begin{equation*}
	\nabla^B:=\nabla+\frac{1}{2}H.
\end{equation*}
We will denote by $\brc=\Ric-\frac{1}{4}H^2$ the symmetric part of its Ricci curvature.

A soliton satisfying \eqref{semi-alg soliton} is said to be \emph{expanding} if $\lambda<0$, \emph{steady} if $\lambda=0$ and \emph{shrinking} if $\lambda >0$. 

We will also say that a GRS is \emph{trivial}, if $\mathcal{L}_Xg=0$, $\mathcal{L}_X H=0$ and $\lambda=0$. In this case, the symmetric part of the Bismut-Ricci tensor is zero, $H$ is harmonic, and $(g,H)$ is Bismut-Ricci flat (BRF) in the sense of \cite{Podest__2022_A}.

	\begin{remark}
		In \cite{garcia2021generalized} a different definition of soliton is given. We remark that such a definition only allows for steady solitons. The definition of soliton that we are adopting is the one given in \cite{Podest__2025,fusi2024homogeneousgeneralizedricciflow}, for which it has been proved  in \cite{fusi2024homogeneousgeneralizedricciflow} that non-trivial expanding examples exist.
	\end{remark}
	We briefly recall how self-similar solutions to the generalized Ricci flow relate to generalized Ricci solitons. 
	
	If $(g_t,H_t)=(c_t\varphi_t^*g_0, c_t\varphi_t^*H_0)$ is a self-similar solution to the GRF for some family $\varphi_t \in \text{Diff}(M)$ and $c_t \in \mathbb{R}$ such that $\varphi_0=\text{Id}$ and $c_0=1$, then:
	\begin{align*}
		-2\Ric^B_{g_0,H_0}=\frac{\partial}{\partial t}g_t|_{t=0}&=c'_0g_0+\mathcal{L}_Xg_0\\
		-\Delta_{g_0}H_0=\frac{\partial}{\partial t}H_t|_{t=0}&=c'_0H_0+\mathcal{L}_XH_0
	\end{align*}
	where $X$ is the vector field on $M$ whose flow is $\varphi_t$. As a result, the pair $(g_0,H_0)$ defines a generalized Ricci soliton.
	
	On the other hand, if $(g,H)$ is a generalized Ricci soliton for a vector field $X$, we recover a self-similar solution to the generalized Ricci flow starting at $(g,H)$ by choosing the simultaneous scaling constant $c_t=1-2\lambda t$, evolving via the flow of the vector field $Y=\frac{1}{c_t}X$.
	
	Let us now turn to homogeneous solitons. A Riemannian manifold $(M,g)$ is said to be homogeneous if  there is a transitive group of isometries $G$. Let $K$ be the isotropy subgroup at a point $o\in M$. Then $M$ can be presented as the coset space $M \cong G/K$, and every automorphism $\Phi$ of $G$ fixing $K$ induces a diffeomorphism $\phi$ of $M$ in the following way:
	\begin{equation*}
		\phi(k\cdot p)=\Phi(k)\cdot p \textbf{ }\textbf{ }\textbf{ }\textbf{ }\forall k\in K,\forall p \in M.
	\end{equation*}
If a self-similar solution to the GRF evolves via diffeomorphisms of $M$ arising in this way, we can give the following definition.
	
	\begin{definition}
		\label{semi-alg GRS}
		A Riemannian $G$-homogeneous space $(M,g_0)$ endowed with a closed, $G$-invariant $3$-form $H_0$ is said to be a \emph{semi-algebraic generalized Ricci soliton} relative to $G$ if it corresponds to a self-similar solution of the GRF evolving via automorphisms of $G$ fixing $K$. More precisely, if there exists a $1$-parameter group $\{\Phi_s\}_{s}$ of automorphisms of $G$ that fix $K$ and constants $c_t$ for which
		\begin{align*}
			g(t)&=c_t\phi_{s(t)}^*g_0\\
			H(t)&=c_t \phi^*_{s(t)}H_0
		\end{align*}
		solves the GRF starting at $(g_0,H_0)$.
	\end{definition}
Let us denote by $\ric$ the Ricci operator  obtained from $\Ric$ after contraction with the metric. Namely,
\begin{equation*}
	\Ric(X,Y)=g(\ric(X),Y) \qquad \forall X,Y \in \Gamma(TM).
\end{equation*}
We shall also write $\ric^B=\ric-\frac{1}{4}g^{-1}H^2$.

If we denote by $\mathfrak{g}$ the Lie algebra of $G$, we see that $GL\mathfrak{(g)}$ acts naturally on the spaces $\Lambda^2\mathfrak{g}^*\otimes \mathfrak{g}$ and $\Lambda^3 \mathfrak{g}^*$ as follows
\begin{align*}
	h\cdot \mu&=g\mu(h^{-1}\cdot, h^{-1}\cdot),\\
	h\cdot \omega&=\omega(h^{-1}\cdot, h^{-1}\cdot , h^{-1}\cdot),
\end{align*}
for every $h\in GL(\mathfrak{g})$.

The corresponding Lie algebra representations are obtained by differentiation of these actions and are given by
\begin{align}
	\label{Lie-alg-rep}
	\theta(E)\mu &:=E\mu-\mu(E\cdot,\cdot)-\mu(\cdot,E\cdot),\\
	\rho(E)\omega&:=-\omega(E\cdot,\cdot,\cdot)-\omega(\cdot,E\cdot,\cdot)-\omega(\cdot,\cdot, E\cdot)\nonumber.
\end{align}
for every $E\in \End(\mathfrak{g})$. These are Lie algebra homomorphisms, namely
\begin{equation}
	\label{Lie-alg-hom}
	\theta([A,B])=[\theta(A),\theta(B)], \qquad  \rho([A,B])=[\rho(A),\rho(B)].
\end{equation}

Clearly semi-algebraic generalized solitons are GRS in the sense of Definition \ref{GRS}, with the advantage that they can be nicely treated in algebraic terms, as shown in the following theorem.
	\begin{theorem}
		Let $(M\cong G/K,g,H)$ be a semi-algebraic generalized Ricci soliton relative to $G$. Then there exists a derivation $D \in \Der(\mathfrak{g})$ that fixes $\mathfrak{k}$ such that on $\mathfrak{g}$
		\begin{equation}
			\label{semi-alg soliton}
			\begin{cases}
				\ric^B=\lambda\Id+ S(D)\\
				\Delta H=2\lambda H-\rho(D)H
			\end{cases}
		\end{equation}
		where $\rho$ is the map on $\End(\mathfrak{g})$ defined in \eqref{Lie-alg-rep}
		and $S(D)=\frac{1}{2}(D+D^T)$.
	\end{theorem}
	\begin{proof}
		Let $\Phi_s$ be the 1-parameter group as in Definition \ref{semi-alg GRS}. Then there exists a derivation $D$ of $\mathfrak{g}$ fixing $\mathfrak{k}$ such that $(d\Phi_s)_e=\exp(-sD)$.  Let $\langle \cdot,\cdot\rangle:=g_{o}$ be the $\operatorname{Ad}(K)$-invariant inner product on $\mathfrak{g}/ \mathfrak{k} \cong T_oM$ induced by the invariant Riemannian metric on $M$. Since $(M,g, H)$ is in particular a GRS we have
		\begin{equation*}
			\langle \ric^B Y, Z \rangle=\lambda \langle Y,Z\rangle-\frac{1}{2}X\cdot\langle Y,Z\rangle    
		\end{equation*}
		where $X$ is the vector field on $G$ with flow $\Phi_s$.
		
		As a result
		\begin{equation*}
			\begin{split}
				X \cdot \langle \cdot,\cdot\rangle 
				&= \frac{d}{ds}\bigg|_{s=0}\Phi_s^* \langle \cdot, \cdot \rangle \\
				&= \frac{d}{ds}\bigg|_{s=0}\left\langle \exp(-sD)\cdot, \exp(-sD)\cdot\right\rangle \\
				&= -\langle D\cdot,\cdot\rangle-\langle\cdot,D\cdot \rangle
			\end{split}
		\end{equation*}
		proving that $\ric^B=\lambda \text{Id}+S(D)$.
		
		\noindent In a similar fashion we get \begin{equation*}
			\begin{split}
				\mathcal{L}_X H 
				&= \frac{d}{ds}\bigg|_{s=0}\Phi_s^*H \\
				&= \frac{d}{ds}\bigg|_{s=0}H\big(\exp(-sD)\cdot,\exp(-sD)\cdot, \exp(-sD)\cdot\big) \\
				&= \rho(D)H
			\end{split}
		\end{equation*}
		so that $\Delta H=2\lambda H-\rho(D)H$, proving the second equation in \eqref{semi-alg soliton}.
	\end{proof}
	
	In case the derivation $D$ is symmetric, the GRS equations simplify and we can consider another important class of solitons on Lie algebras.
	\begin{definition}
		A metric Lie algebra $(\mathfrak{g}, \langle \cdot,\cdot \rangle)$, together with a closed $3$-form $H$, a symmetric derivation $D$ and a real number $\lambda$, is said to be an \emph{algebraic} generalized Ricci soliton if
			\begin{equation}
			\label{algebraic soliton}
			\begin{cases}
				\ric^B=\lambda\text{Id}+ D\\
				\Delta H=2\lambda H-\rho(D)H.
			\end{cases}
		\end{equation}
		In this case, we will denote the algebraic GRS by $(\mathfrak{g}, \mu, \langle\cdot, \cdot \rangle, H, D,\lambda)$.
	\end{definition}

	Such solitons arise as self-similar solutions of the generalized bracket flow introduced in \cite{fusi2024homogeneousgeneralizedricciflow}. It is clear that every semi-algebraic generalized Ricci soliton solving \eqref{semi-alg soliton} is algebraic whenever $D$ is symmetric. In this case, we will say that the Riemannian homogeneous space $(M=G/K,g,H)$ is an algebraic generalized Ricci soliton relative to $G$.
	
	\begin{remark}
	These definitions are given in analogy with the classical algebraic and semi-algebraic Ricci solitons. It is worth pointing out that any classical homogeneous Ricci soliton $(M,g)$ is semi-algebraic relative to $\text{Iso}(M,g)$, as shown in \cite{Jablonski+2015+159+182}. Moreover, every homogeneous Ricci soliton admits a transitive group of isometries $G$ for which $\mathfrak{g}=\Lie(G)$ has a structure of algebraic Ricci soliton, see \cite{Lafuente2014, Jablonski_2014}. In particular, every homogeneous Ricci soliton is algebraic relative to some group $G$.
	\end{remark}
	
	\section{Preliminaries on solvable Lie algebras}
	In this section we recall some useful formulas and facts about solvable metric Lie algebras.
	
	For a given solvable metric Lie algebra $\mathfrak{s}$ we will use the following notation:
	\begin{itemize}
		\item $\mu \in \Lambda^2\mathfrak{s}^*\otimes \mathfrak{s}$ is the Lie bracket of $\mathfrak{s}$.
		\item $\mathfrak{n}$ is the \emph{nilradical} of $\mathfrak{s}$, that is, the unique maximal nilpotent ideal of $\mathfrak{s}$. Since $\mathfrak{n}$ is an ideal, the restriction of $\mu$ to $\mathfrak{n} \times \mathfrak{n}$ is a Lie bracket. We will denote it by $\mu_{\mathfrak{n}}$;
		\item $\mathfrak{a}$ is the orthogonal complement to $\mathfrak{n}$ so that $\mathfrak{s}=\mathfrak{a} \oplus \mathfrak{n}$;
		\item $U$ is the mean curvature vector of $\mathfrak{s}$ defined by
		\begin{equation*}
			\langle U, X \rangle = \Tr(\ad_X) 
		\end{equation*}
		for all $X \in \mathfrak{s}$. 
		Note that $\Tr(\ad_X)=0$ for every $X \in \mathfrak{n}$, meaning that $U \in \mathfrak{a}$ by orthogonality; 
		\item $B_\mu$ is the Killing endomorphism of $(\mathfrak{s}, \mu)$ defined by
		\begin{equation*}
			\langle B_\mu X, Y \rangle = \Tr(\ad_X\ad_Y).
		\end{equation*}
		for all $X,Y \in \mathfrak{s}$. Again, since $\Tr(\ad_X)=0$ for every $X \in \mathfrak{n}$, $B_\mu |_\mathfrak{n}=0$ and $B_\mu \in \End(\mathfrak{a})$;
		\item $M_\mu$ is the endomorphism of $(\mathfrak{s}, \mu)$ defined, using an orthonormal basis $\{X_i\} \subset \mathfrak{s}$, by:
		\begin{equation*}
			\langle M_\mu X, X \rangle = - \frac{1}{2} \sum \langle \mu(X, X_i), X_j \rangle^2 + \frac{1}{4} \sum \langle \mu(X_i, X_j), X \rangle^2
		\end{equation*}
		for all $X \in \mathfrak{s}$.
		\item  $d_\mathfrak{s}$ is the Chevalley-Eilenberg differential of $\mathfrak{s}$. It is defined for every $\omega \in \Lambda^k\mathfrak{s}^*$ by:
		\begin{equation*}
			d_\mathfrak{s}\omega(X_0,\dots,X_k)=\sum_{\substack{i<j\\ i=0}}^k(-1)^{i+j}\omega(\mu(X_i,X_j), X_0, \dots, \hat{X_i}, \dots, \hat{X_j},\dots, X_k) \qquad \forall X_0,\dots,X_k \in \mathfrak{s}.
		\end{equation*} 
		\item $d_\mathfrak{n}$ is the Chevalley-Eilenberg differential on the nilradical $\mathfrak{n}$. It is defined analogously to $d_\mathfrak{s}$, for every differential form in $\Lambda^k\mathfrak{n}^*$. 
	\end{itemize}
	\begin{remark}
		The map $\mu \mapsto M_\mu$ is the moment map for the action of $GL(\mathfrak{s})$ on the space of Lie brackets on $\mathfrak{s}$. A more conceptual definition can be found in \cite{Lauret2006}. There, it is also proved that the moment map may be alternatively defined via the identity
		\begin{equation}
			\Tr(M_\mu A):=\frac{1}{4}\langle\theta(A)\mu,\mu\rangle.
			\label{moment_map}
		\end{equation}
		for every $A \in \End(\mathfrak{s})$.
		We also remark that in \cite{fusi2024homogeneousgeneralizedricciflow} the authors give a generalized formulation of the moment map for the action of the Lie group $\Lambda^2\mathfrak{s}^* \rtimes GL(\mathfrak{s}) $ on the space of Dorfman brackets on $\mathfrak{s} \oplus \mathfrak{s}^*$. 
	\end{remark}

	The inner product on $\mathfrak{s}$ induces an $O(n)$-invariant inner product on $\Lambda^2\mathfrak{s}^*\otimes \mathfrak{s}$ that we still denote by $\langle\cdot, \cdot \rangle $ and is defined by
	\begin{equation*}
		\langle\mu, \nu \rangle=\sum_{i,j} \langle \mu(V_i, V_j), \nu(V_i, V_j) \rangle, \qquad \forall \mu,\nu \in \Lambda^2\mathfrak{s}^*\otimes \mathfrak{s}
	\end{equation*}
	where $\{V_i\}$ is an orthonormal basis of $\mathfrak{s}$.
	Similarly, we have an Ad$(O(n))$-invariant inner product on $\Lambda^k \mathfrak{s}^*$ defined as follows
	\begin{equation*}
		\langle \alpha, \beta \rangle=\sum_{i_1,\dots,i_k}  \alpha(V_{i_1}, \dots, V_{i_k}) \beta(V_{i_1}, \dots, V_{i_k}), \qquad \forall \alpha,\beta \in \Lambda^k\mathfrak{s}^*.
	\end{equation*}
	With respect to these inner products, the representations in \eqref{Lie-alg-rep} satisfy
	\begin{equation}
		\label{rep-trasposte}
		\theta(E)^T=\theta(E^T), \qquad \rho(A)^T=\rho(A^T).
	\end{equation}

	Now fix orthonormal bases $\{A_i\}_{i=1}^a$ for $\mathfrak{a}$ and $\{X_i\}_{i=1}^n$ for $\mathfrak{n}$. The volume forms of $\mathfrak{n},\mathfrak{a},\mathfrak{s}$ are given, respectively, by:
	\begin{equation*}
		\vol_\mathfrak{n}:=X^1\wedge \dots \wedge X^n, \qquad \vol_\mathfrak{a}:=A^1\wedge \dots \wedge A^a, \qquad \vol_\mathfrak{s}=\vol_\mathfrak{n}\wedge \vol_\mathfrak{a}. 
	\end{equation*} 
	Set $s:=a+n$. The \emph{Hodge star operator} $\star_\mathfrak{s} \colon \Lambda^k \mathfrak{s}^* \to \Lambda^{s-k}\mathfrak{s}^*$ on the metric Lie algebra $(\mathfrak{s}, \langle\cdot,\cdot\rangle)$ is defined by
	\begin{equation*}
		 \alpha \wedge \star_\mathfrak{s} \beta= \frac{1}{k!}\langle \alpha, \beta\rangle\vol_\mathfrak{s}, \qquad \forall \alpha,\beta \in \Lambda^{k}\mathfrak{s}^*.
	\end{equation*}
	This operator does not depend on the chosen basis, but only on the fixed volume form and the metric on the Lie algebra. Moreover, it satisfies
	\begin{equation}
		\star_\mathfrak{s}^2=(-1)^{k(s-k)}\Id_{\Lambda^k\mathfrak{s}^*}.
	\end{equation}
	
	The \emph{codifferential} $d^*_\mathfrak{s}$  is given on $k$-forms by
	\begin{equation*}
		d_\mathfrak{s}^*\omega=(-1)^{s(k-1)+1}\star_\mathfrak{s}d_\mathfrak{s}\star_\mathfrak{s}\omega.
	\end{equation*}
	
	The \emph{Laplacian} $\Delta_\mathfrak{s}$ is the self-adjoint operator on $\Lambda^k \mathfrak{s}^*$ defined by 
	\begin{equation*}
	\Delta_\mathfrak{s}=d_\mathfrak{s}d_\mathfrak{s}^*+d_\mathfrak{s}^*d_\mathfrak{s}.
	\end{equation*}
	
	In the following sections we will denote by $\star_\mathfrak{n}, d_\mathfrak{n}^*, \Delta_\mathfrak{n}$ respectively the Hodge star, codifferential and Laplacian of the nilradical $\mathfrak{n}$.

	\section{Solvable generalized Ricci solitons}
	In this section, motivated by the work \cite{Lauret+2011+1+21} on classical solsolitons, we prove some structural properties of generalized Ricci solitons on solvable Lie algebras. From now on, $(\mathfrak{s}, \mu,  \langle\cdot,\cdot\rangle)$ will denote a solvable metric Lie algebra with Lie bracket $\mu$.
    
We recall the formula for the Ricci curvature of left-invariant metrics, from \cite{besse2007einstein}:
\begin{equation*}
	\ric=M_\mu-\frac{1}{2}B_\mu-S(\ad_U).
\end{equation*} 
Then 
\begin{equation}
	\label{Gen ricci curvature}
	\ric^B=M_\mu-\frac{1}{2}B_\mu-S(\ad_U)-\frac{1}{4}H^2
\end{equation}
where we are still writing $H^2$ to denote the endomorphism of $\mathfrak{s}$ obtained by contracting the symmetric form $H^2$ via the inner product.

We can use this expression to obtain formulas similar to \cite[(25)]{Lauret+2011+1+21}. Fix orthonormal bases $\{A_i\}$ for $\mathfrak{a}$ and $\{X_i\}$ for $ \mathfrak{n}$. Then, for any $A\in \mathfrak{a}$ and $X \in \mathfrak{n}$ 
\begin{align}
	\langle \ric^B\, A, A \rangle &= - \frac{1}{2} \sum \| \mu(A, A_i) \|^2 - \Tr \, \left(S(\ad_A|_{\mathfrak{n}})\right)^2-\frac{1}{4}\langle H^2A,A\rangle,\nonumber \\
	\label{eq:lauret formulas} 
	\langle \ric^B\, A, X \rangle &= - \frac{1}{2}\sum \langle \mu(A, A_i), \mu(X, A_i) \rangle - \frac{1}{2} \Tr \, (\ad_A|_{\mathfrak{n}})^T \ad_X|_{\mathfrak{n}}+\nonumber\\
	& - \frac{1}{2} \langle \mu(U, A), X \rangle-\frac{1}{4}\langle H^2A,X\rangle,  \\
	\langle \ric^B\, X, X \rangle &= \frac{1}{4} \sum\langle \mu(A_i, A_j), X \rangle^2 + \frac{1}{2}\sum \langle \left[\ad_{A_i}|_{\mathfrak{n}}, (\ad_ {A_i}|_{\mathfrak{n}})^T\right](X), X \rangle+ \nonumber\\ 
	& - \frac{1}{2}\sum \langle \mu(X, X_i), X_j \rangle^2 + \frac{1}{4} \sum\langle \mu(X_i, X_j), X \rangle^2+\nonumber\\
	& - \langle \mu(U, X), X \rangle-\frac{1}{4}\langle H^2X,X \rangle \nonumber.
\end{align}
In a similar fashion to \cite[Proposition 4.6]{Lauret+2011+1+21} we are able to prove a first structural result for generalized Ricci solsolitons. We start with a technical result regarding the behaviour of closed forms under the orthogonal decomposition $\mathfrak{s}=\mathfrak{a} \oplus \mathfrak{n}$ of the solvable Lie algebra.
\begin{lemma}
	\label{lemma1}
       Let $(\mathfrak{s}=\mathfrak{a} \oplus \mathfrak{n}, \langle \cdot,\cdot\rangle)$ be the orthogonal decomposition of a solvable metric Lie algebra with nilradical $\mathfrak{n}$. Let $H\in \Lambda^k\mathfrak{n}^*$ be a closed $k$-form on $\mathfrak{s}$. Assume that $\mathfrak{a}$ is abelian and acts on $\mathfrak{n}$ via skew-symmetric derivations.
       Then
       \begin{equation*}
       	d_{\mathfrak{s}} \star_{\mathfrak{n}} H=d_{\mathfrak{n}} \star_{\mathfrak{n}}H.
       \end{equation*}
\end{lemma}
\begin{proof}
	 Since $\star_{\mathfrak{n}} H$ is an alternating form on $\mathfrak{n}$ and $\mathfrak{a}$ is abelian, we have
	 \begin{equation*}
	 	d_{\mathfrak{s}} \star_{\mathfrak{n}} H=d_{\mathfrak{n}} \star_{\mathfrak{n}}H +\sum \alpha^i \wedge \rho(\ad_{A_i})\star_{\mathfrak{n}}H
	 \end{equation*}for any orthonormal basis $\{A_i\} \subset \mathfrak{a}$ with dual basis $\{\alpha^i\} \subset \mathfrak{a}^*$. 
	 By the skew-symmetry of $\ad_A$ we have that $\ad_A\in \mathfrak{so}(\mathfrak{n},\langle\cdot,\cdot\rangle)$. In particular, $\rho(\ad_A)\star_{\mathfrak{n}}= \star_{\mathfrak{n}}\rho(\ad_A)$.
	 
     But now since $H$ is a $d_\mathfrak{s}$-closed form on $\mathfrak{n}$ we can conclude that \begin{equation*}
	\rho(\ad_A)H=\mathcal{L}_A H=d_\mathfrak{s}\iota_AH=0
\end{equation*}
and the thesis follows.
\end{proof}

We can now prove the first structural result.
\begin{theorem}[Theorem \ref{A}]
	\label{teoA}
	Let $(\mathfrak{s}=\mathfrak{a} \oplus \mathfrak{n}, \mu ,\langle \cdot, \cdot \rangle, H,D,\lambda)$ be a solvable algebraic generalized Ricci soliton with $H\neq 0$. Then $\lambda<0$.
\end{theorem}
\begin{proof}
	If $\mathfrak{a}=\{0\}$, then $\mathfrak{s}=\mathfrak{n}$ is a nilpotent algebraic GRS. Then $\lambda<0$ by \cite[Proposition 7.7]{fusi2024homogeneousgeneralizedricciflow}.
	
	Thus, we can focus on algebras $\mathfrak{s}=\mathfrak{a}\oplus\mathfrak{n}$ with $\dim(\mathfrak{a})>0$.
	
	Recall that any derivation of a solvable Lie algebra has image contained in the nilradical: $D(\mathfrak{s})\subset \mathfrak{n}$ (see e.g. \cite[Corollary (12.41)]{sagle1986introduction}).
	
Suppose that $\lambda \geq0$. Then, from the first equation in \eqref{eq:lauret formulas} we deduce that
\begin{itemize}
	\item $\lambda=0$;
	\item $[\mathfrak{a},\mathfrak{a}]=0$;
	\item $\ad_A$ is skew-symmetric for every $A\in\mathfrak{a}$;
	\item $H\in \Lambda^3\mathfrak{n}^*$ because $\langle H^2A,A \rangle=\|\imath_AH\|^2=0$ for every $A\in\mathfrak{a}$.
\end{itemize}
The last equation in \eqref{eq:lauret formulas} tells us that $D|_\mathfrak{n}=\ric^B_\mathfrak{n}$ where $\ric^B_\mathfrak{n}=\ric_\mathfrak{n}-\frac{1}{4}H^2$ is the symmetric Bismut-Ricci operator on $(\mathfrak{n}, \langle \cdot, \cdot\rangle_{\mathfrak{n}\times\mathfrak{n}})$.

Now, for $\omega \in \Lambda^3 \mathfrak{n}^* \subset \Lambda^3\mathfrak{s}^*$ it holds that
\begin{align*}
    d_\mathfrak{s}\omega(X_1, X_2,X_3,X_4)&=d_\mathfrak{n}\omega(X_1,X_2,X_3,X_4)\\
    d_\mathfrak{s} \omega(A, X_1,X_2,X_3)&=\rho(\ad_A)\omega(X_1,X_2,X_3)
\end{align*}
for $X_i \in \mathfrak{n}$ and $A \in\mathfrak{a}$. In particular, since $H$ is closed in $\mathfrak{s}$, $d_\mathfrak{n}H=0$. 

At this point, set $a:=\dim(\mathfrak{a})$ and $n:=\dim(\mathfrak{n})$. We observe that
\begin{equation*}
	\star_{\mathfrak{s}}H=\star_{\mathfrak{n}}H \wedge \vol_\mathfrak{a}
\end{equation*}
and by Lemma \ref{lemma1}
\begin{equation*}
	d_\mathfrak{s}\star_{\mathfrak{s}}H=d_{\mathfrak{n}}\star_\mathfrak{n}H \wedge \vol_{\mathfrak{a}} + (-1)^{n-3}\star_{\mathfrak{n}}H \wedge d_\mathfrak{s}\vol_{\mathfrak{a}}.
\end{equation*}
Since $\mathfrak{n}$ is an ideal and $\mathfrak{a}$ is abelian, the second summand is zero, so that
\begin{equation*}
	d^*_\mathfrak{s}H=-\star_{\mathfrak{s}}(d_{\mathfrak{n}}\star_\mathfrak{n}H \wedge \vol_{\mathfrak{a}})=-(-1)^{a[n-(n-2)]}\star_\mathfrak{n}d_\mathfrak{n}\star_\mathfrak{n}H=d^*_\mathfrak{n}H.
\end{equation*}
As a result, we obtain $\Delta_\mathfrak{s}H=\Delta_\mathfrak{n}H$, indeed:
\begin{equation*}
	\Delta_\mathfrak{
	s}H=d_\mathfrak{s}d_\mathfrak{s}^*H=d_\mathfrak{s}d_\mathfrak{n}^*H=d_\mathfrak{n}d_\mathfrak{n}^*H+\sum \alpha^i\wedge \rho(\ad_{A_i})H=\Delta_\mathfrak{n}H
\end{equation*}
where we used the fact that $\imath_A H=0$ for every $A\in \mathfrak{a}$. Thus, the second GRS equation holds.

 Given that $D|_\mathfrak{n}$ is a derivation of $\mathfrak{n}$, this would imply that $(\mathfrak{n}, \mu_{\mathfrak{n}}, \langle \cdot, \cdot\rangle_{\mathfrak{n}\times\mathfrak{n}},  H, D|_\mathfrak{n}, 0)$ is a steady generalized nilsoliton, which cannot happen by \cite[Proposition 7.7]{fusi2024homogeneousgeneralizedricciflow}.
\end{proof}

Next, we turn to the second structural result of this section. From now on we  make three further simplifying assumptions. The first one is that $\mathfrak{s}$ is a unimodular Lie algebra. This means that the mean curvature vector $U$ vanishes.
The second assumption is that $\mathfrak{n}$ is abelian. Finally, we assume that $\mathfrak{a}=\mathbb{R}A$, where we may choose $\|A\|=1$ and set $\alpha=A^\flat\in \mathfrak{a}^*$. The formulas in \eqref{eq:lauret formulas} undergo significant simplifications. 
It is useful to study the behaviour of the 3-form $H$ under the decomposition $\mathfrak{s}=\mathfrak{a}\oplus\mathfrak{n}$.

 Indeed, we can write
 \begin{equation}
 	\label{H-decomp}
 H=\alpha \wedge \omega +\eta
 \end{equation}  where  $\omega \in \Lambda^2\mathfrak{n}^*$ and $\eta \in \Lambda^3\mathfrak{n}^*$. Since $H$ is closed and $\mathfrak{n}$ is abelian,
 \begin{equation*}
 	0=d_\mathfrak{s}H=d_\mathfrak{s}\alpha \wedge \omega-\alpha \wedge d_\mathfrak{n}\omega+d_\mathfrak{n}\eta+\alpha \wedge \rho(\ad_A)\eta=\alpha \wedge \rho(\ad_A)\eta,
 \end{equation*} and in particular
 \begin{equation}
 	\label{rho-ad-A-eta}
 	\rho(\ad_A)\eta=0.
 \end{equation}
 By a slight abuse of notation we will still denote by $\rho$ the Lie algebra representation of the action of $GL(\mathfrak{s})$ on $\Lambda^2\mathfrak{s}^*$. Namely,
 \begin{equation*}
 	\rho(E)\omega=-\omega(E\cdot,\cdot)-\omega(\cdot,E\cdot)
 \end{equation*}
 for every $\omega \in \Lambda^2\mathfrak{s}^*$ and every $E\in \End(\mathfrak{s})$.
 
   Finally, we write $H^2_\mathfrak{n}$ to denote the projection of $H^2|_\mathfrak{n}$ onto $\mathfrak{n}$, that is: 
\begin{equation*}
H^2=\begin{bmatrix} H^2_\mathfrak{a} & H^2_{\mathfrak{an}} \\ (H^2_{\mathfrak{an}})^T & H^2_\mathfrak{n} \end{bmatrix}.
\end{equation*}

We can establish the following result
\begin{lemma}
	\label{Lemma}
	Let $(\mathfrak{s}=\mathfrak{a}\oplus\mathfrak{n}, \mu, \langle \cdot,\cdot\rangle)$ be a solvable unimodular metric Lie algebra, with abelian nilradical $\mathfrak{n}$ and  one-dimensional orthogonal complement $\mathfrak{a}$, say $\mathfrak{a}=\mathbb{R}A$. Assume that
	\begin{itemize}
		\item  $(\langle\cdot,\cdot\rangle,H=\alpha \wedge \omega +\eta)$ defines an algebraic generalized Ricci soliton with $\lambda<0$ on $\mathfrak{s}$, where the decomposition of $H$ is as in \eqref{H-decomp};
		\item $\rho(\ad_A)\omega=0$, that is,  $\omega$ is invariant under the action of $\mathfrak{a}$.
	\end{itemize}
	Then
	\begin{equation*}
		\langle\theta(H^2_\mathfrak{n})\mu, \mu \rangle \leq 0.
	\end{equation*}
\end{lemma}
\begin{proof}
 Consider $E \in  \text{End}(\mathfrak{s})$, such that $E|_\mathfrak{a}=0, E(\mathfrak{n})\subset \mathfrak{n}$. Choose an orthonormal basis $\{X_1,\dots, X_\mathfrak{n},A\}$ of $\mathfrak{s}$ with $\{X_i\}_{i=1}^n$ orthonormal basis of $\mathfrak{n}$. Then, since $H^2_\mathfrak{n}$ is symmetric, we obtain
	\begin{align*}
		\Tr(H^2_\mathfrak{n}E)&=\sum_{i=1}^n \langle H^2_\mathfrak{n}X_i, EX_i \rangle=\sum_{i=1}^n H_\mathfrak{n}^2(X_i,EX_i)\\
		&=\sum_{i,j,k=1}^n H(EX_i,X_j,X_k)H(X_i,X_j,X_k)+ 2\sum_{i,j=1}^n H(EX_i,A,X_j)H(X_i,A,X_j)\\
		&=-\frac{1}{3}\sum_{i,j,k=1}^n \rho(E)\eta(X_i,X_j,X_k)\eta(X_i,X_j,X_k) -\sum_{i,j=1}^n\rho(E)\omega(X_i,X_j)\omega(X_i,X_j)\\
		&=-\frac{1}{3}\langle\rho(E)\eta,\eta\rangle-\langle\rho(E)\omega,\omega\rangle.
	\end{align*} 
	Since $\ad_A|_\mathfrak{n}$ is an endomorphism of $\mathfrak{n}$ we obtain
\begin{align*}
	\langle\theta(H^2_\mathfrak{n})\mu, \mu \rangle &= 2\sum\langle H^2_\mathfrak{n}\mu(A,X_i), \mu(A,X_i)\rangle - 2\sum\langle\mu(A,H^2_\mathfrak{n}X_i),\mu(A,X_i)\rangle\\
	&=2\Tr(\ad_A^TH^2_\mathfrak{n}\ad_A)-2\Tr(\ad_A^T\ad_A H^2_\mathfrak{n})\\
	&=2\Tr(H^2_\mathfrak{n}[\ad_A, \ad_A^T])\\
	&=-\frac{2}{3}\langle\rho([\ad_A, \ad_A^T])\eta, \eta \rangle-2\langle\rho([\ad_A, \ad_A^T])\omega, \omega \rangle\\
	&=-\frac{2}{3}||\rho(\ad_A^T)\eta||^2-2||\rho(\ad_A^T)\omega||^2\leq0
\end{align*}
where in the last equality we used the properties \eqref{Lie-alg-hom}, \eqref{rep-trasposte} of $\rho$, equation \eqref{rho-ad-A-eta} and the hypothesis $\rho(\ad_A)\omega=0$.
\end{proof}
We are now in a position to prove the following.
\begin{proposition}[Proposition \ref{propB}, item (2)]
	\label{teoB}
	Let $(\mathfrak{s}=\mathfrak{a}\oplus\mathfrak{n}, \mu, \langle \cdot,\cdot\rangle, H, \lambda)$ be a solvable unimodular expanding generalized Ricci soliton with abelian nilradical $\mathfrak{n}$ and 1-dimensional orthogonal complement $\mathfrak{a}=\mathbb{R}A$. If $H=\alpha \wedge \omega+\eta$ is such that $\omega$ is $\mathfrak{a}$-invariant, then
	\begin{equation*}
		\ad_A^T \in  \Der(\mathfrak{s}).
	\end{equation*}
\end{proposition}
\begin{proof}
	Using the moment map formula \eqref{moment_map} and the identity \eqref{Gen ricci curvature} we obtain
	\begin{equation}
		\label{momentmapRic}
		\Tr\left(\lambda \Id+D+\frac{1}{4}H^2+\frac{1}{2}B_\mu\right)E=\frac{1}{4}\langle\theta(E)\mu, \mu \rangle
	\end{equation}
for every $E\in \text{End}(\mathfrak{s})$.

Set $F:=D+\frac{1}{4}H^2$. According to the decomposition $\mathfrak{s}=\mathbb{R}A \oplus \mathfrak{n}$ we have 
\begin{align*}
	D &= \begin{bmatrix} 0 & 0 \\ 0 & D_\mathfrak{n} \end{bmatrix} & H^2=\begin{bmatrix}* & * \\ * & H^2_\mathfrak{n} \end{bmatrix}.
\end{align*}
The fact that $D|_\mathfrak{a}=0$ follows from $D(\mathfrak{a})\subset \mathfrak{n}$ and by the symmetry of $D$. If we choose $E_1\in \End(\mathfrak{s})$ given by
\begin{equation*}
	 E_1= \begin{bmatrix} 0 & 0 \\ 0 & \Id_\mathfrak{n} \end{bmatrix}
\end{equation*} in \eqref{momentmapRic}, we get
\begin{equation}
	\label{cond2}
	\lambda n+\Tr(F|_\mathfrak{n})=0, \qquad n=\text{dim}(\mathfrak{n}). 
\end{equation}
where $F_\mathfrak{n}$ is the projection of $F|_\mathfrak{n}$ onto $\mathfrak{n}$.

Now, choose
\begin{equation*}
	 E_2= \begin{bmatrix} 0 & 0 \\ 0 & F_\mathfrak{n} \end{bmatrix}
\end{equation*} in \eqref{momentmapRic} and use Lemma \ref{Lemma} to obtain
\begin{equation}
	\label{cond3}
	\lambda \Tr F_\mathfrak{n}+\Tr F_\mathfrak{n}^2\leq0
\end{equation}
In particular, from \eqref{cond2} and \eqref{cond3} we get $\Tr(F_\mathfrak{n}^2)\leq\lambda^2n$. On the other hand, from Cauchy-Schwarz inequality we also have $\Tr(F_\mathfrak{n}^2)\geq\frac{1}{n}(\Tr F_\mathfrak{n})^2=\lambda^2n$. 

Since we reach equality in Cauchy-Schwarz, then from \eqref{cond2} we see that $F_\mathfrak{n}=-\lambda \Id_\mathfrak{n}$. Thus by looking at the last formula in \eqref{eq:lauret formulas} we obtain
\begin{equation}
	\label{final}
	\lambda \Id +D= \frac{1}{2}[\ad_A, \ad_A^T] -\frac{1}{4}H^2
\end{equation}
which implies $[\ad_A, \ad_A^T]=0$. By  \cite[Lemma 4.7]{Lauret+2011+1+21} we get that $\ad_A^T\in \text{Der}(\mathfrak{s})$.
\end{proof}

To deal with forms with non invariant $\omega$ we will make use of the following algebraic decomposition of the Laplacian
\begin{proposition}[Proposition \ref{propB}, item (1)]
	\label{propLaplaciano}
Let $(\mathfrak{s}=\mathfrak{a}\oplus\mathfrak{n}, \mu, \langle \cdot,\cdot\rangle, H, \lambda)$ be a solvable unimodular expanding generalized Ricci soliton with abelian nilradical $\mathfrak{n}$ and 1-dimensional orthogonal complement $\mathfrak{a}=\mathbb{R}A$. If $H=\alpha \wedge \omega+\eta$, then the condition $\Delta H=2\lambda H-\rho(D)H$ is equivalent to
\begin{equation*}
	\begin{cases}
		 \rho(\ad_A)\left(\rho(\ad_A^T)\omega\right)=2\lambda \omega-\rho(D) \omega\\
		2\lambda\eta=\rho(D)\eta.
	\end{cases}
\end{equation*}
\end{proposition}
\begin{proof}
	 Recall that $\Tr(\ad_A)=0=\Tr(\ad_A|_{\mathfrak{n}})$ by unimodularity. For every $\gamma_1,\gamma_2 \in \Lambda^k\mathfrak{n^*}$, we then have
	\begin{align*}
		0=\rho(\ad_A)(\gamma_1 \wedge \star_\mathfrak{n}\gamma_2)&=\rho(\ad_A)\gamma_1\wedge \star_\mathfrak{n}\gamma_2+\gamma_1\wedge \rho(\ad_A)\star_\mathfrak{n}\gamma_2\\
		&=\langle\gamma_1,\rho(\ad_A^T)\gamma_2\rangle\vol_\mathfrak{n}+\gamma_1\wedge \rho(\ad_A)\star_\mathfrak{n}\gamma_2
    \end{align*}
    and
    \begin{equation*}
		\gamma_1\wedge \rho(\ad_A)\star_\mathfrak{n}\gamma_2 =(-1)^{k(n-k)}\langle \gamma_1, \star_\mathfrak{n}\rho(\ad_A)\star_\mathfrak{n}\gamma_2 \rangle\vol_\mathfrak{n}
	    \end{equation*}	
	so that
	\begin{equation*}
		\rho(\ad_A^T)\gamma=(-1)^{k(n-k)+1}\star_\mathfrak{n}\rho(\ad_A)\star_\mathfrak{n}\gamma \qquad \forall \gamma \in \Lambda^k\mathfrak{n}^*.
	\end{equation*}
Now, since $\mathfrak{n}$ is abelian, we obtain
\begin{align*}
	\star_\mathfrak{s}(\alpha \wedge \omega)&=\star_\mathfrak{n}\omega\\
	d_\mathfrak{s}\star_\mathfrak{s} (\alpha \wedge \omega)&=\alpha \wedge \rho(\ad_A)\star_\mathfrak{n}\omega\\
	\star_\mathfrak{s} d_\mathfrak{s}\star_\mathfrak{s}(\alpha \wedge \omega)&=\star_\mathfrak{n}\rho(\ad_A)\star_\mathfrak{n}\omega=-\rho(\ad_A^T)\omega.
\end{align*}
Thus 
\begin{equation*}
	\Delta (\alpha \wedge \omega)=\alpha \wedge \rho(\ad_A)\left(\rho(\ad_A^T)\omega\right).
\end{equation*}
Similarly, $\star_\mathfrak{s}\eta=-\alpha \wedge \star_\mathfrak{n}\eta$, so that
\begin{equation*}
	d_\mathfrak{s}\star_\mathfrak{s}\eta=\alpha \wedge(\alpha \wedge \rho(\ad_A)\star_\mathfrak{n}\eta)=0
\end{equation*}
so that $d_\mathfrak{s}^*\eta=0$. By equation \eqref{rho-ad-A-eta} we also know that $d_\mathfrak{s}\eta=0$. This implies that $\Delta_\mathfrak{s}\eta=0$ and as a consequence $\Delta H=\alpha \wedge \rho(\ad_A)\left(\rho(\ad_A^T)\omega\right)$.

If $(\langle\cdot,\cdot\rangle,H)$ defines a solsoliton, we then have $\Delta H=2\lambda H-\rho(D)H$, which is equivalent to
\begin{align*}
	\alpha \wedge \rho(\ad_A)\left(\rho(\ad_A^T)\omega\right)&=2\lambda (\alpha \wedge \omega+\eta)-\rho(D)(\alpha \wedge \omega + \eta)\\
	&=2\lambda (\alpha \wedge \omega+\eta)-\alpha\wedge \rho(D)\omega-\rho(D)\eta.
\end{align*}
whence the thesis follows.
\end{proof}

	\section{Algebraic GRS in dimension 3}
	\label{sec_dim3}
	In this section we show that the unique algebraic generalized  Ricci soliton on any non-nilpotent Lie algebra of dimension 3 is given by a Bismut-Ricci flat metric on $\mathfrak{su}(2)$. The nilpotent case is studied in \cite{fusi2024homogeneousgeneralizedricciflow}, while classical solsolitons are classified in \cite{Lauret+2011+1+21}. 
	
	On a 3-dimensional Lie algebra, any 3-form is proportional to the volume form. As a result, on a generalized Ricci soliton $(\langle \cdot,\cdot\rangle,H,D,\lambda)$ the $3$-form $H$ satisfies
	\begin{equation*}
		\Delta H=2\lambda H-\rho(D)H=(2\lambda+\Tr(D))H.
	\end{equation*}
	\begin{theorem}
		\label{class_3d}
	Up to isometry, simultaneous positive scaling of the pair $(\langle \cdot,\cdot\rangle, H)$ and replacement of $H$ with $-H$ , the unique $3$-dimensional, non-nilpotent, algebraic generalized Ricci soliton with nonzero $3$-form $H$ is $(\mathfrak{su}(2),\langle \cdot,\cdot\rangle, H, D=0, \lambda=0)$ where
	\begin{itemize}
		\item the brackets are given on a $\langle \cdot,\cdot\rangle$-orthonormal basis $\{e_1,e_2,e_3\}$ by
		\begin{equation*}
			\mu(e_1,e_2)=e_3, \qquad \mu(e_2,e_3)=e_1, \qquad \mu(e_3,e_1)=e_2;
		\end{equation*}
		\item $H= e^{123}$.
	\end{itemize}
	\end{theorem}
\begin{proof}
	We divide the proof into two parts. First, assume that the Lie algebra $\mathfrak{g}$ is unimodular.
	
	Fix an inner product $\langle \cdot,\cdot\rangle$ on $\mathfrak{g}$. Then, thanks to \cite{MILNOR1976293} it is always possible to find an orthonormal basis $\{e_1,e_2,e_3\}$ for which the only nonzero brackets are
	\begin{equation*}
		\mu(e_2,e_3)=\lambda_1 e_1,  \qquad \mu(e_3,e_1)=\lambda_2 e_2, \qquad \mu(e_1,e_2)=\lambda_3 e_3,
	\end{equation*}
	with $\lambda_1,\lambda_2,\lambda_3$ not all zero. 
	
	With this choice one has that the Ricci operator is given \cite[Theorem 4.3]{MILNOR1976293}, 
	 by:
	\begin{equation*}
		\ric=\frac{1}{2}\text{diag}\left(\lambda_1^2-(\lambda_2-\lambda_3)^2, \lambda_2^2-(\lambda_1-\lambda_3)^2, \lambda_3^2-(\lambda_1-\lambda_2)^2\right).
	\end{equation*}
	A generic closed 3-form $H$ is given by $H=be^{123}$ for some $b\in \mathbb{R}$. Then $H^2=2b^2\text{Id}$ and
	\begin{equation*}
		\ric^B=\ric-\frac{1}{4}H^2=\ric-\frac{b^2}{2}\text{Id}.
	\end{equation*}
	
	 Since $\mathfrak{g}$ is non-nilpotent, at most one of the $\lambda_i$'s may be zero. Consequently we restrict to the cases
	\begin{enumerate}
		\item $\lambda_1\lambda_2\lambda_3 \neq0 $;
		\item $\lambda_i=0$ and $\lambda_j\lambda_k\neq 0$ with $\{i,j,k\}=\{1,2,3\}$.
	\end{enumerate}
	
	If $\lambda_1\lambda_2\lambda_3 \neq0 $, then $D_i^i=0$ for all $i\in\{1,2,3\}$. In particular $\Tr(D)=0$. Moreover the brackets impose $\Delta H=0$ so
	\begin{equation*}
		0=\Delta H=(2\lambda+\Tr(D))H=2\lambda H
	\end{equation*}
	forces $\lambda=0$. As a consequence,
	\begin{equation*}
		\ric^B-\lambda\Id \in \Der(\mathfrak{g}) \Leftrightarrow \ric^B=0
	\end{equation*}
	 This means that that the only Lie algebras in case $(1)$ that admit an algebraic GRS are the ones with $\ric=\frac{1}{4}H^2=\frac{b^2}{2}\Id$. From \cite{MILNOR1976293} we know that the only $3$-dimensional Lie algebra for which the Ricci operator has signature $(+,+,+)$ is $\mathfrak{su}(2)$. This proves that $(\mathfrak{su}(2),\langle \cdot,\cdot\rangle, H, D=0, \lambda=0)$ is a Bismut-Ricci flat GRS, and it is the unique GRS in the case $\lambda_1\lambda_2\lambda_3\neq0$.
	
	If instead, $\lambda_1=0$ and $\lambda_2\lambda_3\neq 0$, then $D_1^1=0$ and $D_2^2=D_3^3=d$. The brackets impose that $\Delta H=0$, so
	\begin{equation*}
		0=\Delta H=(2\lambda+\Tr(D))H=(2\lambda+2d) H.
	\end{equation*}
	This means that $\lambda=-d$. In order that $\ric^B-\lambda \text{Id} \in \text{Der}(\mathfrak{g})$ it is then necessary that $(\ric^B-\lambda\Id)_1^1=0$, that is
	\begin{equation}
		\label{lambda_3dunimodular}
		\lambda=-\frac{1}{2}\left((\lambda_2-\lambda_3)^2+b^2\right).
	\end{equation}
	It is also necessary that $(\ric^B-\lambda\Id)_2^2=(\ric^B-\lambda\Id)_3^3$, meaning:
	\begin{equation*}
		\lambda_2^2-\lambda_3^2= \lambda_3^2-\lambda_2^2.
	\end{equation*}
	This equation is solved if and only if $\lambda_2=\pm \lambda_3$, which immediately implies $\ric_2^2=\ric_3^3=0$.
 Note that we also need $(\ric^B-\lambda\Id)_2^2=d=-\lambda$. Expanding this gives:
	\begin{equation*}
		\ric_2^2 - \frac{b^2}{2} - \lambda = -\lambda \implies \ric_2^2 = \frac{b^2}{2}.
	\end{equation*}
Since we just established that $\ric_2^2 = 0$, this forces $\frac{b^2}{2} = 0$, meaning $b=0$. We can conclude that there do not exist any GRS with nonzero $H$ in this case.

    Now, let us assume that the Lie algebra is non-unimodular. Fix an inner product $\langle \cdot,\cdot \rangle$ on $\mathfrak{g}$. Then, by \cite{MILNOR1976293}, $\mathfrak{g}$ admits a $\langle \cdot,\cdot \rangle$-orthonormal frame $\{e_1,e_2,e_3\}$ for which
	\begin{equation*}
		\mu(e_1,e_2)=\alpha e_2+\beta e_3, \qquad \mu(e_1,e_3)=\gamma e_2+ \delta e_3 \qquad \mu(e_2,e_3)=0.
	\end{equation*}
	with $\alpha+\delta \neq 0$ and $\alpha\gamma+\beta\delta=0$. This basis  is often called a \emph{Milnor frame} for the Lie algebra. It diagonalizes the Ricci tensor by
	\begin{equation*}
		\ric=\text{diag}\left(-\alpha^2-\delta^2-\frac{1}{2}(\beta+\gamma)^2,-\alpha(\alpha+\delta)+\frac{1}{2}(\gamma^2-\beta^2), -\delta(\alpha+\delta)+\frac{1}{2}(\beta^2-\gamma^2)\right)
	\end{equation*}
	Consider a generic 3-form $H=b e^{123}$, $b\in\mathbb{R}$. Then $H^2=2b^2\text{Id}$ and
	\begin{align*}
		(\ric^B)_1^1&=-\alpha^2-\delta^2-\frac{1}{2}(\beta+\gamma)^2-\frac{b^2}{2},\\
		(\ric^B)_2^2&=-\alpha(\alpha+\delta)+\frac{1}{2}(\gamma^2-\beta^2)-\frac{b^2}{2},\\
		(\ric^B)_3^3&=-\delta(\alpha+\delta)+\frac{1}{2}(\beta^2-\gamma^2)-\frac{b^2}{2}.
	\end{align*}
	To determine the space of derivations, we recall that $D\in \text{Der}(\mathfrak{g})$ if and only if:
	\begin{align*}
		D\mu(e_1,e_2)&=\mu(De_1,e_2)+\mu(e_1,De_2),\\
		D\mu(e_1,e_3)&=\mu(De_1,e_3)+\mu(e_1, De_3),\\
		0&=\mu(De_2,e_3)+\mu(e_2,De_3).
	\end{align*}
	From the first two conditions we obtain, in particular, that
	\begin{equation*}
		\begin{cases}
			\beta D_3^2= \alpha D_1^1+\gamma D_2^3\\
			\gamma D_2^3=\delta D_1^1+\beta D_2^3
		\end{cases}
	\end{equation*}
	which implies that $D_1^1=0$, since $\alpha+\delta\neq0$. Given that $\ric^B$ is diagonal and $(\ric^B)_1^1<0$, we see that
	\begin{equation*}
	\ric^B \notin \text{Der}(\mathfrak{g}).
	\end{equation*}
	This proves that steady algebraic solitons for the generalized Ricci flow cannot exist on 3-dimensional solvable, non-unimodular Lie algebras.
	
	More generally, if one has $\ric^B-\lambda \text{Id} \in \text{Der}(\mathfrak{g})$, then
	\begin{equation}
		\label{lambda_3d_nonunimod}
		\lambda=-\alpha^2-\delta^2-\frac{1}{2}(\beta+\gamma)^2-\frac{b^2}{2}.
	\end{equation}
	Then again, $\Delta H=0$ and from $\Tr(D)=-2\lambda$
	\begin{equation*}
		D=\ric-\frac{1}{4}H^2-\lambda \Id \Rightarrow -2\lambda=\operatorname{scal}-\frac{1}{4}\Tr(H^2)-3\lambda \Rightarrow \lambda=\operatorname{scal}-\frac{1}{4}\Tr(H^2).
	\end{equation*}
	Substituting \eqref{lambda_3d_nonunimod} yields
	\begin{equation*}
		-\alpha^2-\delta^2-\frac{1}{2}(\beta+\gamma)^2-\frac{b^2}{2}=-2\alpha^2-2\delta^2-2\alpha\delta-\frac{1}{2}(\beta+\gamma)^2-\frac{3}{2}H^2 \Rightarrow (\alpha+\delta)^2+b^2=0.
	\end{equation*}
	Since $\alpha+\delta\neq0$, it follows that we cannot find any non-classic solvable soliton having these structure equations.
	\end{proof}

	\section{Algebraic solvable GRS in dimension 4}
\label{sec_dim4}

	In this section we turn to the 4-dimensional case, with the complete classification of solvable, unimodular, non-nilpotent algebraic generalized Ricci solitons that are non-classical.  The classification of generalized Ricci solitons for nilpotent Lie algebras up to dimension 4 has been provided in \cite{fusi2024homogeneousgeneralizedricciflow} whereas the classical Ricci solsolitons are classified in \cite{Lauret+2011+1+21}. There are eight solvable, unimodular, non-nilpotent Lie algebras of dimension 4 up to isomorphism (see e.g. \cite{Andrada2005} for the classification). We split the analysis into two cases based on whether the nilradical is abelian or non-abelian.
	
\subsection{Algebras with abelian nilradical}

A solvable unimodular Lie algebra with 3-dimensional abelian nilradical is isomorphic to one of those listed in Table \ref{table}. For each Lie algebra, the nilradical is spanned by $\{e_1,e_2,e_3\}$ and the action of a fourth basis vector $e_4$ is given, for each algebra, in the second column. Thanks to Propositions \ref{teoB} and \ref{propLaplaciano} we are able to give a complete classification of non-classical generalized solsolitons on each of these Lie algebras. We are adopting the notation used in \cite{Andrada2005}.

\begin{remark}
	For this class of Lie algebras and for suitable $H$, the constant $\lambda$ is entirely determined by the trace of $D|_\mathfrak{n}$. This is because, writing $H=\alpha \wedge \omega +\eta$ as in Proposition \ref{propLaplaciano} then, if $\eta \neq 0$, we have that $\eta$ is proportional to the volume form of $\mathfrak{n}$. So, $\rho(D)\eta=-\Tr(D|_\mathfrak{n})\eta$ which implies, by Proposition \ref{propLaplaciano}, that $2\lambda=-\Tr(D|_\mathfrak{n})$.
\end{remark}

	\setlength{\tabcolsep}{20pt}
\begin{table}[H]
	\centering
	\caption{Algebras with abelian nilradical}
	\begin{tabular}{lcl}
		\toprule
		$\mathfrak{s}$ & \qquad & $\ad_{e_4}|_\mathfrak{n}$ \\
		\midrule
		$\mathfrak{r}_{4,-1/2,-1/2}$ & \qquad& $\begin{bmatrix} 1 & 0 & 0 \\ 0 & 1 & 0 \\ 0 & 0 & -2 \end{bmatrix}$ \\ \addlinespace[10pt]
		$\mathfrak{e}(1,1) \oplus \mathbb{R}$ & \qquad& $\begin{bmatrix} 0 & 0 & 0 \\ 0 & 1 & 0 \\ 0 & 0 & -1 \end{bmatrix}$ \\ \addlinespace[10pt]
		$\mathfrak{r}'_{4,-2\mu,\mu} \qquad (\mu>0)$ & \qquad & $\begin{bmatrix} \mu & 1 & 0 \\ -1 & \mu & 0 \\ 0 & 0 & -2\mu \end{bmatrix}$ \\ \addlinespace[10pt]
		$\mathfrak{r}_{4,\mu,-1-\mu} \qquad (\mu>1)$ & \qquad & $\begin{bmatrix} \mu & 0 & 0 \\ 0 & 1 & 0 \\ 0 & 0 & -1-\mu \end{bmatrix}$ \\ \addlinespace[10pt]
		$\mathfrak{r}_{4,-1/2}$ & \qquad & $\begin{bmatrix} 1 & 1 & 0 \\ 0 & 1 & 0 \\ 0 & 0 & -2 \end{bmatrix}$ \\ \addlinespace[10pt]
		$ \mathfrak{e}(2)\oplus \mathbb{R}$ & \qquad &$\begin{bmatrix} 0 & 1 & 0 \\ -1 & 0 & 0 \\ 0 & 0 & 0 \end{bmatrix}$ \\
		\bottomrule
	\end{tabular}
	\label{table}
\end{table}
	
	\begin{proposition}
		\label{class1}
		On the Lie algebra $\mathfrak{r}_{4,-1/2,-1/2}$ the unique generalized Ricci soliton with nonzero $H$, up to simultaneous positive scaling of the pair $(\langle \cdot,\cdot\rangle, H)$ and replacement of $H$ with $-H$, is $(\mathfrak{r}_{4,-1/2,-1/2}, \langle\cdot,\cdot\rangle, H, D, \lambda)$, where:
		\begin{itemize}
			\item  the bracket relations are, with respect to a $\langle\cdot,\cdot\rangle$-orthonormal basis $X_1, \dots, X_4$ having dual basis $X^1, \dots, X^4$:
			\begin{equation*}
				\mu(X_4, X_1)=X_1, \qquad \mu(X_4, X_2)=X_2, \qquad \mu(X_4,X_3)=-2X_3;
			\end{equation*} 
			\item $H=2X^{123}$;
			\item $D =\diag(4,4,4,0)$;
			\item $\lambda=-6$.	
		\end{itemize}
	\end{proposition}
	\begin{proof}
		$\mathfrak{r}_{4,-1/2,-1/2}$ is the Lie algebra spanned by $e_1,\dots,e_4$ with non-zero Lie brackets:
		\begin{equation*}
			\mu(e_4,e_1) = e_1 \qquad
			\mu(e_4,e_2) = e_2  \qquad
			\mu(e_4,e_3) = -2e_3.
		\end{equation*}
		Thanks to \cite[Theorem 4.1]{VanThuong2017} every metric on this algebra is isometric to one defined by an orthonormal basis $X_1,\dots,X_4$ and constants $a>0, b\geq0$ for which the bracket relations are:
		\begin{equation*}
			\mu(X_4,X_1) = aX_1, \qquad
			\mu(X_4,X_2) = aX_2, \qquad
			\mu(X_4,X_3) = 3abX_1 - 2aX_3.
		\end{equation*}
		Then, a computation shows that the Ricci operator is represented by the following matrix, with respect to the basis $\{X_1, \dots, X_4\}$:
		\begin{equation*}
			\ric=\begin{bmatrix}
				\frac{9 a^{2} b^{2}}{2} & 0 & -\frac{9 a^{2} b}{2} & 0 
				\\
				0 & 0 & 0 & 0 
				\\
				-\frac{9 a^{2} b}{2} & 0 & -\frac{9 a^{2} b^{2}}{2} & 0 
				\\
				0 & 0 & 0 & -\frac{9}{2} a^{2} b^{2}-6 a^{2} 
			\end{bmatrix}.
		\end{equation*}
		Let $X^1,\dots,X^4$ be the corresponding dual basis and $H=\lambda_4 X^{123}+\lambda_3 X^{124}+\lambda_2 X^{134}+\lambda_1 X^{234}$ a generic 3-form on $\mathfrak{g}$. Using the splitting in Proposition \ref{propLaplaciano} we write $H=X^4\wedge \omega+\eta$, where $\omega=\lambda_3 X^{12}+\lambda_2 X^{13}+\lambda_1 X^{23}$ and $\eta=\lambda_4 X^{123}$. Using the structure equations it is easy to check that $dH=0$. Then the symmetric 2-form $H^2$ has the following structure: 
		\begin{equation}
			\label{H_squared}
			H^2 = 2 \begin{bmatrix}
				\lambda_{4}^{2} + \lambda_{3}^{2} + \lambda_{2}^{2} & \lambda_{2} \lambda_{1} & -\lambda_{3} \lambda_{1} & \lambda_{4} \lambda_{1} \\[6pt]
				\lambda_{2} \lambda_{1} & \lambda_{4}^{2} + \lambda_{3}^{2} + \lambda_{1}^{2} & \lambda_{3} \lambda_{2} & -\lambda_{4} \lambda_{2} \\[6pt]
				-\lambda_{3} \lambda_{1} & \lambda_{3} \lambda_{2} & \lambda_{4}^{2} + \lambda_{2}^{2} + \lambda_{1}^{2} & \lambda_{4} \lambda_{3} \\[6pt]
				\lambda_{4} \lambda_{1} & -\lambda_{4} \lambda_{2} & \lambda_{4} \lambda_{3} & \lambda_{3}^{2} + \lambda_{2}^{2} + \lambda_{1}^{2}
			\end{bmatrix}
		\end{equation}
		A generic derivation of the Lie algebra is:
		\begin{equation*}
			D=\begin{bmatrix}
				m_{11} & m_{1 2} & b \left(m_{1 1}-m_{3 3}\right) & m_{1 4} 
				\\
				m_{2 1} & m_{2 2} & b m_{2 1} & m_{2 4} 
				\\
				0 & 0 & m_{3 3} & m_{3 4} 
				\\
				0 & 0 & 0 & 0 
			\end{bmatrix}
		\end{equation*}
	meaning that we get symmetric derivations if and only if $m_{14} =m_{24}=m_{34}=0$ and either $b \neq 0, m_{11}=m_{33}$ and $m_{21}=m_{12}=0$ or $b=0, m_{21}=m_{12}$. 
	
		Assume first that $\lambda_1=\lambda_2=\lambda_3=0$. Then, from $\ric^B=\lambda\Id+D$ we obtain $b=0$ and $m_{11}=m_{22}=m_{33}=-\frac{\lambda_4^2}{2} +6a^2$. The soliton equation $\Delta H=2\lambda H-\rho(D)H$ now imposes
		\begin{equation*}
			12 a^{2} \lambda_4-3\lambda_4 \left(-\frac{\lambda_4^{2}}{2}+6 a^{2}\right)=0
		\end{equation*}
	which is solved if and only if $\lambda_4=\pm 2a$ ($\lambda_4=0$ leads to $H=0$). As a result this is the only GRS with $\lambda_1=\lambda_2=\lambda_3=0$.
	
	Assume now that at least one of $\lambda_1,\lambda_2,\lambda_3$ is not zero. Then, the first equation in \eqref{algebraic soliton} forces $\lambda_4=0$.
Computing $\langle \ric^B(X_i), X_i\rangle$ and using again the first equation in \eqref{algebraic soliton} yields
\begin{align*}
	\lambda &= -\frac{9}{2}a^2b^2 - 6a^2 - \frac{1}{2}\lambda_3^2 - \frac{1}{2}\lambda_2^2 - \frac{1}{2}\lambda_1^2, \\
	m_{33} &= 6a^2 + \frac{1}{2}\lambda_3^2, \\
	m_{22} &= \frac{9}{2}a^2b^2 + 6a^2 + \frac{1}{2}\lambda_2^2,\\
	m_{11} & =9a^2b^2+ 6a^2 +\frac{1}{2}\lambda_1^2.
\end{align*}

 Assume first that the symmetric derivation is given by $m_{11}=m_{33}$ and $m_{21}=m_{12}=0$. Then, from $\langle \ric^B(X_1), X_2\rangle$ and $\langle \ric^B(X_2), X_3\rangle$ we obtain that either $\lambda_2=0$ or $\lambda_3=\lambda_1=0$. In the former case, $\rho(D)\omega=-(m_{33}+m_{22})\omega$ and in the latter, $\rho(D)\omega=-2m_{33}\omega$. In particular 
 \begin{equation*}
 	2\lambda+m_{33}+m_{22}=-\frac{9}{2}a^2b^2-\frac{\lambda_3^2}{2}-\lambda_1^2\leq0
 \end{equation*} 
 in the first case and
  \begin{equation*}
  2\lambda+2m_{33}=-9a^2b^2-\lambda_2^2\leq0
 \end{equation*} 
 in the second case.
  Since the equation $\rho(\ad_{X_4})\rho(\ad_{X_4}^T)\omega=(2\lambda-\rho(D))\omega$ holds from Proposition \ref{propLaplaciano}, in both cases we would get that the positive semi-definite operator $\rho(\ad_{X_4})\rho(\ad_{X_4}^T)$ has a nonpositive eigenvalue. Then it must be zero and in both cases this reduces to $H=0$.

If instead, $b=0, m_{21}=m_{12}$, then $m_{21}=-\frac{\lambda_2\lambda_1}{2}$ and either $\lambda_3=0$ or $\lambda_1=\lambda_2=0$. If $\lambda_3=0$ the GRS equation $\Delta H=2\lambda H-\rho(D)H$ implies 
\begin{equation*}
	\begin{cases}
		\lambda_2(a^2 + \lambda_2^2 + \lambda_1^2)=0,\\ 
		\lambda_1(a^2 + \lambda_2^2 + \lambda_1^2)=0,
	\end{cases}
\end{equation*}
meaning that $\lambda_1=\lambda_2=0$, so that $H=0$.

If, instead, $\lambda_1=\lambda_2=0$, the second GRS equation is
\begin{equation*}
	\lambda_3\left(-8a^2 - 2\left(-6a^2 - \frac{\lambda_3^2}{2}\right)\right)=0
\end{equation*}
which again, implies $\lambda_3=0$ and $H=0$.
	\end{proof}
	
	\begin{proposition}
	On the Lie algebra $\mathfrak{e}(1,1) \oplus \mathbb{R}$ the unique generalized Ricci soliton with nonzero $H$, up to simultaneous positive scaling of the pair $(\langle \cdot,\cdot\rangle, H)$ and replacement of $H$ with $-H$,  is $(\mathfrak{e}(1,1)\oplus \mathbb{R}, \langle\cdot,\cdot\rangle, H, D, \lambda)$, where:
		\begin{itemize}
			\item the bracket relations are, with respect to a $\langle \cdot,\cdot \rangle $-orthonormal basis $X_1, \dots, X_4$ having dual basis $X^1, \dots, X^4$: 
			\begin{equation*}
				\mu(X_2, X_4)=-X_2, \qquad \mu(X_3, X_4)=X_3;
			\end{equation*} 
			\item $H=\frac{2}{\sqrt{3}}X^{123}$;
			\item $D=\diag\left(\frac{4}{3},\frac{4}{3},\frac{4}{3},0\right)$;
			\item $\lambda=-2$.
		\end{itemize}
	\end{proposition}
	\begin{proof}
		The algebra $\mathfrak{g}:=\mathfrak{e}(1,1) \oplus \mathbb{R}$ is the Lie algebra spanned by $e_1,\dots, e_4$ with non-zero Lie brackets:
		\begin{equation*}
			\mu(e_4,e_2) = e_2, \qquad
			\mu(e_4,e_3) = -e_3.
		\end{equation*}
		Thanks to \cite[Theorem 4.4]{VanThuong2017} every metric on this Lie algebra is isometric to one with orthonormal basis $X_1 \dots X_4$, corresponding dual basis $X^1,\dots, X^4$ and brackets:
		\begin{align*}
			\mu(X_2,X_4) &= b_{12}b_{44}X_1 - b_{44}X_2, \\
			\mu(X_3,X_4) &= b_{44}(2b_{12}b_{23}-b_{13})X_1 - 2b_{23}b_{44}X_2 + b_{44}X_3,
		\end{align*}
		where $b_{44}>0, b_{12}, b_{23}\geq0, b_{13}\in \mathbb{R}$.
		For these brackets the Ricci operator is represented by the following matrix, with respect to the basis $\{X_1, \dots, X_4\}$:
		\begin{equation*}
			\ric=\frac{ b_{44}^2}{2}
			\begin{bmatrix}
				b_{12}^2 + (2b_{12}b_{23} - b_{13})^2 & -(4b_{12}b_{23}^2 - 2b_{13}b_{23} + b_{12}) & 2b_{12}b_{23} - b_{13} & 0 \\
				-(4b_{12}b_{23}^2 - 2b_{13}b_{23} + b_{12}) & -b_{12}^2 + 4b_{23}^2 & -(2b_{12}^2 b_{23} - b_{12}b_{13} + 4b_{23}) & 0 \\
				2b_{12}b_{23} - b_{13} & -(2b_{12}^2 b_{23} - b_{12}b_{13} + 4b_{23}) & -(2b_{12}b_{23} - b_{13})^2 - 4b_{23}^2 & 0 \\
				0 & 0 & 0 & *
			\end{bmatrix}
		\end{equation*}
	with last diagonal entry
		\begin{equation*}
			\ric_4^4 = -\frac{b_{12}^{2} b_{44}^{2}}{2}-\frac{b_{44}^{2} (2 b_{12} b_{23}-b_{13})^{2}}{2}-2 b_{23}^{2} b_{44}^{2}-2 b_{44}^{2}.
		\end{equation*}
		As before let $H=\lambda_4 X^{123}+\lambda_3 X^{124}+\lambda_2 X^{134}+\lambda_1 X^{234}$ be a generic 3-form on $\mathfrak{g}$. Then the endomorphism $H^2$ has the same form as in \eqref{H_squared}. Also,  using the splitting in Proposition \ref{propLaplaciano} we write $H=X^4\wedge \omega+\eta$, where $\omega=\lambda_3 X^{12}+\lambda_2 X^{13}+\lambda_1 X^{23}$ and $\eta=\lambda_4 X^{123}$.
		
		A generic derivation for this Lie algebra is of the form:
		\begin{equation*}
			D = \begin{bmatrix}
				m_{11} & (m_{11}-m_{22}) b_{12} & -b_{23} (m_{22}-m_{33}) b_{12}+b_{13} (m_{11}-m_{33}) & m_{14} \\[6pt]
				0 & m_{22} & (m_{22}-m_{33}) b_{23} & m_{24} \\[6pt]
				0 & 0 & m_{33} & m_{34} \\[6pt]
				0 & 0 & 0 & 0 
			\end{bmatrix}
		\end{equation*}
	so that we get symmetric derivations if $m_{14}=m_{24}=m_{34}=0$ and
	\begin{itemize}
		\item $b_{12}=b_{23}=b_{13}=0$;
		\item $b_{12}=b_{23}=0$ and $m_{11} =m_{33}$;
		\item $m_{11}=m_{22}=m_{33}$.
	\end{itemize}
	We will analyse the cases  $m_{11}=m_{22}=m_{33}$ and $b_{12}=b_{23}=b_{13}=0$. The remaining one can be reduced to one of these two.
		
		In the first case, assume that $\lambda_1=\lambda_2=\lambda_3=0$. Then the first equation in \eqref{algebraic soliton} implies that $b_{13}=b_{23}=b_{12}=0, \lambda=-2b_{44}^2, m_{33}=-\frac{\lambda_4^2}{2}+2b_{44}^2$. At this point the GRS equation $\Delta H=2\lambda H-\rho(D)H$ becomes
		\begin{equation*}
			4b_{44}^2\lambda_4 - 3\lambda_4\left(2b_{44}^2 -\frac{\lambda_4^2}{2}\right)=0
		\end{equation*}
		which is solved if and only if $\lambda_4=\pm \frac{2}{\sqrt{3}}b_{44}$. This is the only non-classical soliton that we get in this situation.
		
		Assume now that at least one of $\lambda_1,\lambda_2,\lambda_3$ is nonzero. This implies $\lambda_4=0$, from $\ric^B=\lambda\Id+D$. Then, computing $\langle\ric^B(X_3),X_3 \rangle$ and $\langle\ric^B(X_4),X_4 \rangle$ yields
		\begin{align*}
			\lambda &= \frac{\left((-4b_{12}^2 - 4)b_{23}^2 + 4b_{12}b_{13}b_{23} - b_{12}^2 - b_{13}^2 - 4\right)b_{44}^2}{2} - \frac{\lambda_3^2}{2} - \frac{\lambda_2^2}{2} - \frac{\lambda_1^2}{2},\\
			m_{33} & = \frac{(b_{12}^2 + 4)b_{44}^2}{2} + \frac{\lambda_3^2}{2}.
		\end{align*}
	In particular,
	 \begin{align*}
		2\lambda+2m_{33}&=((-4b_{12}^2 - 4)b_{23}^2 + 4b_{12}b_{13}b_{23} - b_{13}^2)b_{44}^2 - \lambda_2^2 - \lambda_1^2=\\
		&=-(2b_{12}b_{23}-b_{13})^2-4b_{23}^2 - \lambda_2^2 - \lambda_1^2\leq 0.
	\end{align*} Using this inequality, the fact that $\rho(D)\omega=-2m_{33}\omega$, and the formula  $\rho(\ad_{X_4})\rho(\ad_{X_4}^T)\omega=(2\lambda-\rho(D))\omega$ from Proposition \eqref{propLaplaciano}, we conclude that the positive semi-definite operator $\rho(\ad_{X_4})\rho(\ad_{X_4}^T)$ has a nonpositive eigenvalue. This can only happen if this eigenvalue is zero, that is only if $\lambda_1=\lambda_2=b_{23}=b_{13}=0$. But once again, computing $\langle\ric^B(X_1),X_2 \rangle$  and using the first equation in \eqref{algebraic soliton} forces $b_{12}=0$, given that $b_{44}>0$. Likewise, $\langle\ric^B(X_2),X_2 \rangle$ forces $\lambda_3=0$. As a consequence, $H=0$.
	
	If, instead, we assume $b_{12}=b_{23}=b_{13}=0$, then $\lambda_4=0$ (otherwise we recover the GRS that we already found in the previous case) and the equation $\Delta H=2\lambda H-\rho(D)H$ implies $\lambda_1=\lambda_2=\lambda_3=0$ meaning that $H=0$.
	\end{proof}

\begin{proposition}
	On the Lie algebra $\mathfrak{r}'_{4,-2\mu,\mu}$ the unique generalized Ricci soliton with nonzero $H$, up to simultaneous positive scaling of the pair $(\langle \cdot,\cdot\rangle, H)$ and replacement of $H$ with $-H$, is $(\mathfrak{r}'_{4,-2\mu,\mu},\langle\cdot,\cdot\rangle, H, D, \lambda )$, where:
	\begin{itemize}
		\item  the bracket relations are, with respect to a $\langle\cdot,\cdot\rangle$-orthonormal basis $X_1, \dots, X_4$ having dual basis $X^1, \dots, X^4$:
		\begin{equation*}
			\mu(X_1, X_4)=-\mu X_1+X_2, \qquad \mu(X_2, X_4)=-X_1-\mu X_2, \qquad \mu(X_3, X_4)=2\mu X_3;
		\end{equation*} 
		\item $H= 2\mu X^{123}$;
		\item $D = \diag(4,4,4,0)\mu^2$;
		\item $\lambda=-6\mu^2$.
	\end{itemize}
\end{proposition}
\begin{proof}
	This is the Lie algebra defined by brackets
	\begin{equation*}
		\mu(e_4,e_1) =\mu e_1-e_2, \qquad
		\mu(e_4,e_2) =e_1+\mu e_2, \qquad
		\mu(e_4,e_3) =-2\mu e_3,
	\end{equation*}
	depending on the positive parameter $\mu$.
	
	Thanks to \cite{VanThuong2017}, every metric on $\mathfrak{g}:=\mathfrak{r}'_{4,-2\mu,\mu}$ belongs to one among three isometry classes. The first one is defined by an orthonormal basis $X_1,\dots,X_4$, with dual basis $X^1,\dots, X^4$ and brackets
	\begin{align*}
		\mu(X_1,X_4) &= -\mu b_{44}X_1+{b_{44}}X_2,\\
		\mu(X_2,X_4) &= -b_{44}X_1-\mu b_{44}X_2,\\
		\mu(X_3,X_4) &= -3\mu b_{44} b_{13}X_1+b_{44}b_{13} X_2+2\mu b_{44}X_3.
	\end{align*}
	and constants $b_{44}>0, b_{13}\geq0$.
	For these brackets, the Ricci operator is represented by the following matrix, with respect to the basis $\{X_1,\dots,X_4\}$:
	\begin{equation*}
		\ric=\frac{b_{44}^2}{2}
		\begin{bmatrix}
			9 \mu^{2} b_{13}^{2} & -3 \mu b_{13}^{2} & -b_{13} \left(9 \mu^{2}+1\right) & 0 \\
			-3 \mu b_{13}^{2} & b_{13}^{2} & 0 & 0 \\
			-b_{13} \left(9 \mu^{2}+1\right) & 0 & -b_{13}^{2} \left(9 \mu^{2}+1\right) & 0 \\
			0 & 0 & 0 &-\left(9 \mu^{2} b_{1 3}^{2}+12 \mu^{2}+b_{1 3}^{2}\right)
		\end{bmatrix}.
	\end{equation*}
	
	A generic derivation of the Lie algebra is of the form:
	\begin{equation*}
		D= \left[\begin{array}{cccc}
			m_{22} & -m_{21} & b_{13} \left(m_{22}-m_{33}\right) & m_{14} 
			\\
			m_{21} & m_{22} & b_{13} m_{21} & m_{24} 
			\\
			0 & 0 & m_{33} & m_{34} 
			\\
			0 & 0 & 0 & 0 
		\end{array}\right]
	\end{equation*}
	so that we have a symmetric derivation if $m_{21}=0$, $m_{14}=m_{24}=m_{34}=0$ and if either $m_{22}=m_{33}$ or $b_{13}=0$. 
	
	Let $H=\lambda_4 X^{123}+\lambda_3 X^{124}+\lambda_2 X^{134}+\lambda_1 X^{234}$ be a generic $3$-form on $\mathfrak{g}$. Then, the endomorphism $H^2$ has the same form as in \eqref{H_squared}. Also, using the splitting in Proposition \ref{propLaplaciano} we write $H=X^4\wedge \omega+\eta$, where $\omega=\lambda_3 X^{12}+\lambda_2 X^{13}+\lambda_1 X^{23}$ and $\eta=\lambda_4 X^{123}$.
	
	First, let us analyze the case $m_{22}=m_{33}$. Assume that $\lambda_1=\lambda_2=\lambda_3=0$. Then, computing $\langle \ric^B(X_1), X_2\rangle$ and using the first equation in \eqref{algebraic soliton}  forces $b_{13}=0$. Similarly, from the computation of $\langle \ric^B(X_4), X_4\rangle$ and $\langle \ric^B(X_3), X_3\rangle$ we deduce that $\lambda=-6\mu^2b_{44}^2, m_{33}=-\frac{\lambda_4^2}{2}+6\mu^2b_{44}^2$. Now the GRS equation $\Delta H=2\lambda H-\rho(D)H$ forces $\lambda_4=\pm 2\mu b_{44}$, meaning that this is the only non-classical soliton we get in this situation.
	
	Assume now that at least one among $\lambda_1,\lambda_2,\lambda_3$ is nonzero. Then, the first equation in \eqref{algebraic soliton} forces $\lambda_4=0$. Moreover, computing $\langle \ric^B(X_4), X_4\rangle$ and $\langle \ric^B(X_3), X_3\rangle$ yields
	\begin{align*}
		\lambda &= \frac{\left((-9\mu^2 - 1)b_{13}^2 - 12\mu^2\right)b_{44}^2}{2} - \frac{\lambda_3^2}{2} - \frac{\lambda_2^2}{2} - \frac{\lambda_1^2}{2},\\
		m_{33} &= 6\mu^2 b_{44}^2 + \frac{\lambda_3^2}{2}.
	\end{align*} 
In particular, 
\begin{equation*}
	2\lambda+2m_{33}=(-9\mu^2-1)b_{44}^2b_{13}^2-\lambda_2^2-\lambda_1^2\leq 0
\end{equation*} and since $\rho(D)\omega=-2m_{33}\omega$, in the equation $\rho(\ad_{X_4})\rho(\ad_{X_4}^T)\omega=(2\lambda-\rho(D))\omega$ we get that a positive semi-definite operator has a nonpositive eigenvalue. This can only happen if this eigenvalue is zero, that is only if $b_{13}=\lambda_1=\lambda_2=0$. But now, $\langle \ric^B(X_1), X_1\rangle$  forces $\lambda_3=0$, thus $H=0$.
	
	If instead $b_{13}=0$, then $\lambda_4=0$ (otherwise we recover the GRS that we already found in the previous case). But then computing $\langle \ric^B(X_1),X_1\rangle$, $\langle \ric^B(X_2),X_2\rangle$ and applying the first equation in \eqref{algebraic soliton} we discover that $\lambda_1=\lambda_2=0$. Equation $\Delta H=2\lambda H-\rho(D)H$ now forces $\lambda_3=0$, thus $H=0$.

	A similar reasoning can be applied to the remaining isometry classes. The second one is defined by brackets
	\begin{align*}
		\mu(X_1,X_4) &= -\mu b_{44}X_1+\frac{b_{44}}{b_{22}}X_2,\\
		\mu(X_2,X_4) &= -b_{22}b_{44}X_1-\mu b_{44}X_2,\\
		\mu(X_3,X_4) &= b_{44}(-3\mu b_{13}-b_{23})X_1+\frac{b_{44}}{b_{22}}(-3\mu b_{23}+b_{13})X_2+2\mu b_{44}X_3.
	\end{align*}
	and constants $b_{44}>0, b_{13},b_{23}\geq0, 0<b_{22}<1$.
	The third isometry class is defined by the same brackets but with $b_{13}>0, b_{23}<0, b_{44}>0, 0<b_{22}<1$.
	
	Since these isometry classes share their bracket relations, a generic derivation takes the same form in both classes:
	\begin{equation*}
		D=\begin{bmatrix}
	m_{22} & 0 & b_{13} \left(m_{22}-m_{33}\right) & 0 
	\\
	0 & m_{22} & \frac{b_{23} \left(m_{22}-m_{33}\right)}{b_{22}} & 0 
	\\
	0 & 0 & m_{33} & 0 
	\\
	0 & 0 & 0 & 0 
		\end{bmatrix}
	\end{equation*}
	and is symmetric if either $b_{13}=b_{23}=0$ (only possible for the second isometry class) or $m_{22}=m_{33}$.
	
	In the case $m_{22}=m_{33}$, following similar computations as before, we get
	\begin{equation*}
		\begin{split}
			\lambda = -\frac{b_{44}^2}{2b_{22}^2}&( 9\mu^2 b_{13}^2 b_{22}^2 + 6\mu b_{13} b_{22}^2 b_{23} + 12\mu^2 b_{22}^2 + 9\mu^2 b_{23}^2 + b_{22}^4+ \\
			&+ b_{22}^2 b_{23}^2 - 6\mu b_{13} b_{23} + b_{13}^2 - 2b_{22}^2 + 1 ) - \frac{\lambda_3^2}{2} - \frac{\lambda_2^2}{2} - \frac{\lambda_1^2}{2}
		\end{split}
	\end{equation*}
and $D=\diag(m_{33},m_{33},m_{33},0)$, where 
\begin{equation*}
m_{33} = \frac{b_{44}^2}{2b_{22}^2} \left( 12\mu^2 b_{22}^2 + b_{22}^4 - 2b_{22}^2 + 1 \right) - \frac{\lambda_4^2}{2} + \frac{\lambda_3^2}{2}.
\end{equation*}
 Thus,
\begin{equation*}
	2\lambda + 2m_{33} = -\frac{b_{44}^2}{b_{22}^2} \left[ b_{22}^2 (3\mu b_{13} + b_{23})^2 + (3\mu b_{23} - b_{13})^2 \right] - \lambda_4^2 - \lambda_2^2 - \lambda_1^2 \leq 0.
\end{equation*}
This inequality, together with the fact that $\rho(D)\omega=-2m_{33}\omega$ and the formula from Proposition \ref{propLaplaciano},  $\rho(\ad_{X_4})\rho(\ad_{X_4}^T)\omega=(2\lambda-\rho(D))\omega$ allow us to conclude that $\lambda_1=\lambda_2=\lambda_4=0$, $b_{23}=-3\mu b_{13}$, $b_{13}=3\mu b_{23}$. The only possibility is $b_{13}=b_{23}=0$, but then, computing $\langle \ric^B(X_2),X_2\rangle$ and using that $\ric^B-\lambda\Id\in \Der(\mathfrak{g})$ forces $b_{44}=0$ or $b_{22}=1$ and both cases are not allowed.

In the case $b_{13}=b_{23}=0$, suppose first that $\lambda_1=\lambda_2=\lambda_3=0$. Then $\langle \ric^B(X_1),X_1\rangle$ means that $b_{22}^2-1=0$ but we require $b_{22}\in (0,1)$. So at least one among $\lambda_1,\lambda_2,\lambda_3$ is nonzero. This forces $\lambda_4=0$. Moreover, computing the off-diagonal terms $\langle \ric^B(X_i),X_j\rangle$ with $i,j\in\{1,2,3\}$ and using the first equation in \eqref{algebraic soliton} we deduce that two terms among $\lambda_1,\lambda_2,\lambda_3$ are zero.
\begin{itemize}
	\item if $\lambda_1=\lambda_3=0$ from $\langle \ric^B(X_1),X_1\rangle$ we obtain that:
	\begin{equation*}
		2b_{22}^2b_{44}^2-b_{22}^2\lambda_2^2-2b_{44}^2=0.
	\end{equation*}
	This second-degree polynomial equation admits no real solution with $b_{22}\in (0,1)$;
	\item if $\lambda_2=\lambda_3=0$ from the second equation in \eqref{algebraic soliton} we obtain
	\begin{equation*}
		\mu b_{44}^2\lambda_1(b_{22}^2 - 1)=0.
	\end{equation*}
	Hence, we do not find solutions with $\lambda_1 \neq 0$, since $b_{22} \in (0,1)$.
	\item if $\lambda_1=\lambda_2=0$, from $\langle \ric^B(X_1),X_1\rangle$ we obtain that $b_{22}^2-1=0$, impossible to solve for $b_{22}\in(0,1)$;
\end{itemize}
\end{proof}

\begin{proposition}
	On the Lie algebra $\mathfrak{r}_{4,\mu,-1-\mu}$ the unique generalized Ricci soliton with nonzero $H$, up to simultaneous positive scaling of the pair $(\langle \cdot,\cdot\rangle, H)$ and replacement of $H$ with $-H$, is $(\mathfrak{r}_{4,\mu,-1-\mu},\langle\cdot,\cdot\rangle, H, D, \lambda )$, where:
	\begin{itemize}
		\item the  bracket relations are, with respect to a $\langle\cdot,\cdot\rangle$-orthonormal basis $X_1, \dots, X_4$ having dual basis $X^1, \dots, X^4$:
		\begin{equation*}
			\mu(X_1, X_4)=-\mu X_1, \quad \mu(X_2, X_4)=-X_2, \quad \mu(X_3, X_4)=(1+\mu)X_3;
		\end{equation*} 
		\item $H=\frac{2}{3}\sqrt{3\mu^2+3\mu+3} X^{123}$;
		\item $D = \frac{4}{3}(\mu^2+\mu+1)\diag(1,1,1,0)$;
		\item $\lambda=(-2\mu^2-2\mu-2)$.
	\end{itemize}
\end{proposition}
	\begin{proof}
	The algebra $\mathfrak{g}:=\mathfrak{r}_{4,\mu,-1-\mu}$ is the Lie algebra defined by bracket relations:
	\begin{equation*}
		\mu(e_4,e_1) = \mu e_1, \qquad
		\mu(e_4,e_2) = e_2, \qquad
		\mu(e_4,e_3) =(-1-\mu)e_3.
	\end{equation*}
	Thanks to \cite[Theorem 4.4]{VanThuong2017} every metric on this Lie algebra is isometric to one with orthonormal basis $X_1 \dots X_4$, corresponding dual basis $X^1,\dots, X^4$ and brackets:
	\begin{align*}
		&\mu(X_1,X_4)=-b_{44}\mu X_1, \qquad  \qquad
		\mu(X_2,X_4)= b_{12}b_{44}(1-\mu)X_1 - b_{44}X_2, \\
		&\mu(X_3,X_4) = b_{44}(b_{12}b_{23}(2 + \mu) - b_{13}(1 + 2\mu))X_1 +b_{23}b_{44}(-\mu-2)X_2 + b_{44}(1+\mu)X_3.
	\end{align*}
	where $b_{44}>0, b_{12}, b_{23}\geq0, b_{13}\in \mathbb{R}$.
	For these brackets the Ricci operator is the symmetric endomorphism with the following nonzero entries:
	\begin{align*}\ric_1^1 &=\frac{b_{44}^2 }{2}\left[ \left(\mu^2 b_{23}^2 + 4\mu b_{23}^2 + \mu^2 + 4b_{23}^2 - 2\mu + 1\right)b_{12}^2 - (2 + 4\mu)(2 + \mu)b_{12}b_{13}b_{23} + 4\left(\frac{1}{2} + \mu\right)^2 b_{13}^2\right], \\
		\ric_1^2 &= -\frac{b_{44}^2 }{2}\left(\left((2 + \mu)^2 b_{23}^2 + (\mu - 1)^2\right)b_{12} - 2\left(\frac{1}{2} + \mu\right)(2 + \mu)b_{13}b_{23}\right), \\
		\ric_1^3 &= b_{44}^2\left(\frac{1}{2} + \mu\right)\left(\mu(b_{12}b_{23} - 2b_{13}) + 2b_{12}b_{23} - b_{13}\right),\\
		\ric_2^2 &= -\frac{b_{44}^2 }{2}\left(\mu(b_{12} + b_{23}) - b_{12} + 2b_{23}\right)\left(\mu(b_{12} - b_{23}) - b_{12} - 2b_{23}\right),\\
		\ric_2^3 &=\frac{b_{44}^2 }{2}\left[ (2 + \mu)(\mu - 1)b_{12}^2 b_{23} + (-2\mu^2 + \mu + 1)b_{12}b_{13} - (2 + \mu)^2 b_{23}\right],\\
		\ric_3^3 &= -\frac{b_{44}^2 }{2}\left((b_{12}^2 + 1)(2 + \mu)^2 b_{23}^2 - (2 + 4\mu)(2 + \mu)b_{12}b_{13}b_{23} + 4\left(\frac{1}{2} + \mu\right)^2 b_{13}^2\right),\\
		\ric_4^4 &=
		-\mu^2 b_{44}^2 - \frac{b_{12}^2 b_{44}^2 (-1 + \mu)^2}{2} - \frac{b_{44}^2 (b_{12} b_{23} (2 + \mu) - b_{13} (1 + 2\mu))^2}{2} \\&- b_{44}^2 - \frac{b_{23}^2 b_{44}^2 (2 + \mu)^2}{2} - b_{44}^2 (1 + \mu)^2.
	\end{align*}
Let $H=\lambda_4 X^{123}+\lambda_3 X^{124}+\lambda_2 X^{134}+\lambda_1 X^{234}$ be a generic $3$-form on $\mathfrak{g}$. Then, the endomorphism $H^2$ has the same form as in \eqref{H_squared}. Also, using the splitting in Proposition \ref{propLaplaciano} we write $H=X^4\wedge \omega+\eta$, where $\omega=\lambda_3 X^{12}+\lambda_2 X^{13}+\lambda_1 X^{23}$ and $\eta=\lambda_4 X^{123}$.

	A generic derivation for this Lie algebra is of the form:
	\begin{equation*}
		D =
		\begin{bmatrix}
			m_{11} & (m_{11}-m_{22}) b_{12} & -b_{23} (m_{22}-m_{33}) b_{12}+b_{13} (m_{11}-m_{33}) & m_{14} \\
			0 & m_{22} & b_{23} (m_{22}-m_{33}) & m_{24} \\
			0 & 0 & m_{33} & m_{34} \\
			0 & 0 & 0 & 0 
        \end{bmatrix}.
	\end{equation*}
In order that $D$ be symmetric it is necessary that $m_{14}=m_{24}=m_{34}=0$. We will analyse the case $m_{11}=m_{22}=m_{33}$ and the case $b_{13}=b_{12}=b_{23}=0$. All the other cases that yield symmetric derivations either reduce to subcases of these two or are solved analogously.
	
  	In the case that $m_{11}=m_{22}=m_{33}$, assume $\lambda_1=\lambda_2=\lambda_3=0$. Then, by computing $\langle \ric^B(X_4),X_4\rangle$, $\langle \ric^B(X_3),X_3\rangle$ and using the first equation in \eqref{algebraic soliton} we obtain:
  	\begin{align*}
  		\lambda &=-\mu^2b_{44}^2-\frac{b_{12}^2b_{44}^2(-1+\mu)^2}{2}-\frac{b_{44}^2(b_{12}b_{23}(2+\mu)-b_{13}(1+2\mu))^2}{2}\\
  		&- \frac{b_{44}^2-b_{23}^2b_{44}^2(2+\mu)^2}{2}-b_{44}^2(1+\mu)^2,\\
  		m_{33} &= \frac{(\mu^2 b_{12}^2 - 2\mu b_{12}^2 + 4\mu^2 + b_{12}^2 + 4\mu + 4)b_{44}^2}{2} - \frac{\lambda_4^2}{2}.
  	\end{align*}
  	Applying the same logic, we compute $\langle \ric^B(X_1),X_3\rangle$ and use equation \eqref{algebraic soliton} to conclude that $b_{13}=\frac{b_{12} b_{23} (2 + \mu)}{1 + 2\mu}$. But now, $\langle \ric^B(X_1),X_2\rangle$ and $\langle \ric^B(X_2),X_3\rangle$ imply that $b_{23}=b_{12}=0$. This means that $b_{13}=0$ and $\lambda=-2b_{44}^2(\mu^2+\mu+1)$.
  	 At this point the GRS equation $\Delta H=2\lambda H-\rho(D)H$ implies that $\lambda_4=\pm \frac{2}{3}\sqrt{3\mu^2+3\mu+3}b_{44}$, so that $m_{33}=\frac{4}{3}b_{44}^2(\mu^2+\mu +1))$. It is easy to check that the equations \eqref{algebraic soliton} hold for these values of $\lambda,\lambda_4, m_{33}$. This proves that we found the only non-classical GRS in this case.
  	 
  	As a result, we can assume that at least one among $\lambda_1,\lambda_2,\lambda_3$ is nonzero. Computing $\langle \ric^B(X_i),X_4\rangle$ for $i\in{1,2,3}$, we see that this implies $\lambda_4=0$. Then, computing $\langle \ric^B(X_4),X_4\rangle$ and $\langle \ric^B(X_3),X_3\rangle$ and using the first equation in \eqref{algebraic soliton} we deduce 
  	\begin{align*}
  		\lambda &= -\mu^2 b_{44}^2 - \frac{b_{12}^2 b_{44}^2 (-1 + \mu)^2}{2} - \frac{b_{44}^2 \left(b_{12}b_{23}(2 + \mu) - b_{13}(1 + 2\mu)\right)^2}{2} - b_{44}^2\\
  		& - \frac{b_{23}^2 b_{44}^2 (2 + \mu)^2}{2} - b_{44}^2 (1 + \mu)^2 - \frac{\lambda_3^2}{2} - \frac{\lambda_2^2}{2} - \frac{\lambda_1^2}{2},\\
  		m_{33} & = \frac{\left(\mu^2 b_{12}^2 - 2\mu b_{12}^2 + 4\mu^2 + b_{12}^2 + 4\mu + 4\right)b_{44}^2}{2} + \frac{\lambda_3^2}{2}.
  	\end{align*}
  In particular,
\begin{equation*}
	2\lambda + 2m_{33}=-b_{44}^2 \left( \left[(\mu+2)b_{12}b_{23}-(2\mu + 1)b_{13} \right]^2+(\mu+2)^2 b_{23}^2 \right)-\lambda_2^2-\lambda_1^2 \leq 0
\end{equation*}
   and since $\rho(D)X^{ij}=-2m_{33}X^{ij}$, if we write $\omega:=\sum \lambda_i X^{jk}$, in the equation $\rho(\ad_{X_4})\rho(\ad_{X_4}^T)\omega=(2\lambda-\rho(D))\omega$ we would get that a positive semi-definite operator has a nonpositive eigenvalue. This can only happen if it is zero, that is only if $\lambda_1=\lambda_2=b_{23}=b_{13}=0$. But now, computing $\langle \ric^B(X_1),X_2\rangle$, $\langle \ric^B(X_1),X_1\rangle$ and using the first equation in \eqref{algebraic soliton} we see that $b_{12}=0$ and  $\lambda_3=0$. This implies $H=0$.
  
  In the case that $b_{13}=b_{12}=b_{23}=0$, then $\lambda_4=0$ (otherwise we go back to the previous case). But then $\Delta H=2\lambda H-\rho(D)H$ forces $\lambda_1=\lambda_2=\lambda_3=0$, meaning $H=0$.
\end{proof}

	\begin{proposition}
	The Lie algebra $\mathfrak{r}_{4,-1/2}$ does not admit any algebraic generalized Ricci soliton with $H\neq 0$.
\end{proposition}
\begin{proof}
	$\mathfrak{g}=\mathfrak{r}_{4,-1/2}$ is the Lie algebra defined by
	\begin{equation*}
		\mu(e_4,e_1)=e_1, \qquad
		\mu(e_4,e_2)=e_1+e_2, \qquad
		\mu(e_4,e_3)=-2e_3.
	\end{equation*}
	Every metric on $\mathfrak{g}$ is isometric to one defined by an orthonormal basis $X_1,\dots,X_4$ and brackets
	\begin{align*}
		\mu(X_1,X_4)&=-b_{44}X_1,\\
		\mu(X_2,X_4)&=-b_{22}b_{44}X_1-b_{44}X_2,\\
		\mu(X_3,X_4)&=b_{44}(-3b_{13}-b_{23})X_1-3\frac{b_{23}b_{44}}{b_{22}}X_2+2b_{44}X_3,
	\end{align*}
	with $b_{22},b_{44}>0, b_{13}\geq0, b_{23}\in \mathbb{R}$.
	
	Let $H=\lambda_4 X^{123}+\lambda_3 X^{124}+\lambda_2 X^{134}+\lambda_1 X^{234}$ be a generic $3$-form on $\mathfrak{g}$. Then, the endomorphism $H^2$ has the same form as in \eqref{H_squared}. Also, using the splitting in Proposition \ref{propLaplaciano} we write $H=X^4\wedge \omega+\eta$, where $\omega=\lambda_3 X^{12}+\lambda_2 X^{13}+\lambda_1 X^{23}$ and $\eta=\lambda_4 X^{123}$.
	
	A generic derivation of this Lie algebra is of the form:
	\begin{equation*}
		D= \begin{bmatrix}
			m_{22} & m_{12} & \frac{m_{22} b_{22} b_{13}-m_{33} b_{22} b_{13}+m_{12} b_{23}}{b_{22}} & m_{14} \\[6pt]
			0 & m_{22} & \frac{b_{23} (m_{22}-m_{33})}{b_{22}} & m_{24} \\[6pt]
			0 & 0 & m_{33} & m_{34} \\[6pt]
			0 & 0 & 0 & 0 
		\end{bmatrix}
	\end{equation*}
so that we only have symmetric derivations for $m_{12}=0$ and either $m_{22}=m_{33}$ or $b_{23}=b_{13}=0$.
Then $\mathfrak{g}= \langle X_1,X_2,X_3\rangle \oplus X_4\mathbb{R}$ where the first summand is the abelian nilradical $\mathfrak{n}$. Now note that
\begin{equation*}
	\ad_{X_4}|_{\mathfrak{n}}=\begin{bmatrix} 
		   b_{44} & b_{22}b_{44} & b_{44}(3b_{13}+b_{23}) \\
		   0 & b_{44} & 3b_{23}b_{44}b_{22}^{-1} \\
		   0 & 0 & -2b_{44} 
	  \end{bmatrix}
\end{equation*}
In view of Proposition \ref{teoB} we do not have solitons for $H=\lambda_4 X^{123}$, since $\ad_{X_4}^T \in \Der(\mathfrak{g})$ can only happen if $b_{22}b_{44}=0$, which is impossible. In particular we cannot have $\lambda_1=\lambda_2=\lambda_3=0$.

The Ricci operator is represented by the following matrix, with respect to the basis $\{X_1,\dots, X_4\}$:
\begin{equation*}
\frac{b_{44}^2}{2}
\begin{bmatrix}
	b_{22}^{2} + \left(3 b_{13}+b_{23}\right)^{2} & 3 \left(3 b_{13}+b_{23}\right) b_{23} b_{22}^{-1} & -3 \left(3 b_{13}+b_{23}\right) & 0 \\
	3 \left(3 b_{13}+b_{23}\right) b_{23} b_{22}^{-1} & -\left(b_{22}^{4}-9 b_{23}^{2}\right) b_{22}^{-2} & -3 \left(\left(b_{13}+\frac{b_{23}}{3}\right) b_{22}^{2}+3 b_{23}\right) b_{22}^{-1} & 0 \\
	-3 \left(3 b_{13}+b_{23}\right) & -3 \left(\left(b_{13}+\frac{b_{23}}{3}\right) b_{22}^{2}+3 b_{23}\right) b_{22}^{-1} & -9 \left(\left(b_{13}+\frac{b_{23}}{3}\right)^{2} b_{22}^{2}+b_{23}^{2}\right) b_{22}^{-2} & 0 \\
	0 & 0 & 0 & \ast
\end{bmatrix}
\end{equation*}
where the last diagonal entry is $\ric^4_4=-12 - b_{22}^{2} - \left(3 b_{13}+b_{23}\right)^{2} - \frac{9 b_{23}^{2}}{b_{22}^{2}}$.
In order that $\ric-\lambda\Id -\frac{1}{4}H^2$ be a derivation, the last row forces $\lambda_4=0$ because we established that at least one of $\lambda_1,\lambda_2,\lambda_3$ is nonzero.

In case $m_{22}=m_{33}$, we can make use of Proposition \ref{propLaplaciano}, recalling that $\rho(\ad_{X_4})\rho(\ad_{X_4}^T) \omega=(2\lambda-\rho(D))\omega$. In particular, note that $\rho(D)\omega=-2m_{22}\omega$. By computing $\langle \ric^B(X_4),X_4 \rangle$, $\langle \ric^B(X_3),X_3 \rangle$ and using the first equation in \eqref{algebraic soliton} we deduce that
\begin{align*}
	\lambda &= -6 b_{44}^{2} - \frac{b_{22}^{2} b_{44}^{2}}{2} - \frac{b_{44}^{2} \left(3 b_{13} + b_{23}\right)^{2}}{2} - \frac{9 b_{23}^{2} b_{44}^{2}}{2 b_{22}^{2}} - \frac{\lambda_{3}^{2}}{2} - \frac{\lambda_{2}^{2}}{2} - \frac{\lambda_{1}^{2}}{2} \\
	m_{22} &= \frac{\left(b_{22}^{2} + 12\right) b_{44}^{2}}{2} + \frac{\lambda_{3}^{2}}{2}.
\end{align*}
 Since $\rho(\ad_{X_4})\rho(\ad_{X_4}^T)$ is a positive semi-definite operator and
 \begin{equation*}
 	2\lambda+2m_{22}=\frac{\left[-9\left(b_{13} + \frac{b_{23}}{3}\right)^2 b_{44}^2 -\lambda_2^2 - \lambda_1^2\right]b_{22}^2 - 9b_{23}^2 b_{44}^2}{b_{22}^2}\leq 0
 \end{equation*} 
 we deduce that this eigenvalue must be zero and that happens only if $\lambda_1=\lambda_2=b_{23}=b_{13}=0$. But now, computing $\langle \ric^B(X_1),X_1 \rangle$ and $\langle \ric^B(X_2),X_2 \rangle$ it is easy to see that we do not find any solution with positive $b_{22},b_{44}$.
 
 	If instead, $b_{23}=b_{13}=0$, then $\lambda_4=0$ (otherwise we reduce to the previous case) and the equation $\Delta H=2\lambda H-\rho(D)H$ implies $\lambda_1=\lambda_2=\lambda_3=0$ meaning that $H=0$.
\end{proof}

\begin{proposition}
	The Lie algebra $\mathfrak{e}(2) \oplus \mathbb{R}$ does not admit any algebraic generalized Ricci soliton with $H\neq 0$.
\end{proposition}
\begin{proof}  $\mathfrak{g}=\mathfrak{e}(2) \oplus \mathbb{R}$ is the Lie algebra defined by
	\begin{equation*}
		\mu(e_4,e_1) =-e_2, \qquad
		\mu(e_4,e_2) =e_1.
	\end{equation*}
	There are two isometry classes for metrics on $\mathfrak{g}$. The first one admits an orthonormal basis $X_1,\dots,X_4$ for which
	\begin{equation*}
		\mu(X_1,X_4) =\frac{b_{44}}{b_{22}}X_2, \qquad
		\mu(X_2,X_4) =-b_{22}b_{44}X_1, \qquad
		\mu(X_3,X_4) =-b_{23}b_{44}X_1+\frac{b_{13}b_{44}}{b_{22}}X_2.
	\end{equation*}
	where $0<b_{22}<1, b_{44}>0, b_{13},b_{23}\geq0$. 
	
	A generic derivation is of the form
	\begin{equation*}D=
	\begin{bmatrix}
		m_{22} & -m_{21}b_{22}^2 & (m_{22} - m_{33})b_{13} - b_{22}b_{23}m_{21} & m_{14} \\
		m_{21} & m_{22} & \frac{b_{13}b_{22}m_{21} + b_{23}(m_{22} - m_{33})}{b_{22}} & m_{24} \\
		0 & 0 & m_{33} & m_{34} \\
		0 & 0 & 0 & 0
	\end{bmatrix}
	\end{equation*}
and we have symmetric derivations if 
$m_{21}=0$, $m_{14}=m_{24}=m_{34}=0$ and either $m_{22}=m_{33}$ or $b_{23}=b_{13}=0$.

The Ricci operator is represented by the following matrix, with respect to the basis $\{X_1,\dots,X_4\}$:
\begin{equation*}
	\frac{b_{44}^2}{2} \begin{bmatrix}
		\left(b_{22}^{4}+b_{23}^{2} b_{22}^{2}-1\right) b_{22}^{-2} & -b_{23} b_{13} b_{22}^{-1} & -b_{13} b_{22}^{-2} & 0 \\
		-b_{23} b_{13} b_{22}^{-1} & \left(-b_{22}^{4}+b_{13}^{2}+1\right) b_{22}^{-2} & -b_{23} b_{22} & 0 \\
		-b_{13} b_{22}^{-2} & -b_{23} b_{22} & -\left(b_{23}^{2} b_{22}^{2}+b_{13}^{2}\right) b_{22}^{-2} & 0 \\
		0 & 0 & 0 & \ast
	\end{bmatrix}
\end{equation*}
where the last entry is 
\begin{equation*}
	\left(-b_{22}^{4}+\left(-b_{23}^{2} +2 \right) b_{22}^{2}-b_{13}^{2} -1\right) b_{22}^{-2}.
\end{equation*}

The algebra decomposes as $\mathfrak{g}= \langle X_1,X_2,X_3\rangle \oplus X_4\mathbb{R}$ where the first summand is the abelian nilradical $\mathfrak{n}$.

 Let $H=\lambda_4 X^{123}+\lambda_3 X^{124}+\lambda_2 X^{134}+\lambda_1 X^{234}$ be a generic $3$-form on $\mathfrak{g}$. Then, the endomorphism $H^2$ has the same form as in \eqref{H_squared}. Also, using the splitting in Proposition \ref{propLaplaciano} we write $H=X^4\wedge \omega+\eta$, where $\omega=\lambda_3 X^{12}+\lambda_2 X^{13}+\lambda_1 X^{23}$ and $\eta=\lambda_4 X^{123}$.
 
  Now note that
\begin{equation*}
	\ad_{X_4}|_{\mathfrak{n}}=\begin{bmatrix} 
		0 & b_{22}b_{44} & b_{44}b_{23} \\
		-b_{44}b_{22}^{-1} & 0 & -b_{13}b_{44}b_{22}^{-1} \\
		0 & 0 & 0 
	\end{bmatrix}
\end{equation*}
In view of Proposition \ref{teoB} we cannot find solutions for $H=\lambda_4X^{123}\in \Lambda^3\mathfrak{n}^*$. Indeed, $\ad_{X_4}^T \in \Der(\mathfrak{g})$ can only happen if $-b_{44}b_{22}^{-1}=-b_{22}^{3}b_{44}$ and $b_{13}=b_{23}=0$, in particular if $b_{22}^4=1$, which is not allowed, since $b_{22}\in (0,1)$.

Therefore, we can assume that at least one of $\lambda_1,\lambda_2,\lambda_3$ is non-zero. Thus, the last row of $\ric-\frac{1}{4}H^2-\lambda \Id$ forces $\lambda_4=0$.

Suppose first that $m_{21}=0$ and $m_{22}=m_{33}$. Then, by computing $\langle \ric^B(X_4),X_4\rangle$, $\langle \ric^B(X_3),X_3\rangle$ and using the first equation in \eqref{algebraic soliton} we deduce that
\begin{align*}
	\lambda &= -\frac{b_{44}^2}{2b_{22}^2}\left(b_{22}^4 + (b_{23}^2 - 2)b_{22}^2 + b_{13}^2 + 1\right)-\frac{\lambda_1^2+\lambda_2^2+\lambda_3^2}{2}\\
	m_{33} &= \frac{b_{22}^{4} b_{44}^{2} + b_{22}^{2} \lambda_{3}^{2} - 2 b_{22}^{2} b_{44}^2 + b_{44}^2}{2 b_{22}^{2}}.
\end{align*}
In this case $\rho(D)\omega=-2m_{33}\omega$ and we notice that
\begin{equation*}
	2\lambda+2m_{33}=\frac{(-b_{23}^2 b_{44}^2 -\lambda_2^2 -\lambda_1^2)b_{22}^2 - b_{13}^2 b_{44}^2}{b_{22}^2}\leq 0.
\end{equation*} 
 Since $\rho(\ad_{X_4})\rho(\ad_{X_4}^T)$ is a positive semi-definite operator we deduce that this nonpositive eigenvalue must be zero, thus $\lambda_1=\lambda_2=b_{13}=b_{23}=0$. But now by computing $\langle \ric^B(X_1), X_1\rangle$, $\langle \ric^B(X_2), X_2\rangle$, using the first equation in \eqref{algebraic soliton} and summing the resulting equations we see that $\lambda_3=0$ and $b_{44}^2(b_{22}^4-1)=0$. The latter equation cannot be solved since $b_{44}>0$ and $b_{22} \in (0,1)$.

 	If instead, $b_{23}=b_{13}=0$, then $\lambda_4=0$ (otherwise we reduce to the previous case) and the equation $\Delta H=2\lambda H-\rho(D)H$ implies $\lambda_1=\lambda_2=\lambda_3=0$ meaning that $H=0$.

	The second isometry class admits a basis for which
	\begin{equation*}
		\mu(X_1,X_4) = b_{44}X_2,\qquad
		\mu(X_2,X_4) = -b_{44}X_1,\qquad
		\mu(X_3,X_4) = b_{13}b_{44}X_2.
	\end{equation*}
A generic derivation for these brackets is of the form:
	\begin{equation*}
		M = \begin{bmatrix}
			m_{22} & -m_{21} & (m_{22}-m_{33}) b_{13} & m_{14} \\[6pt]
			m_{21} & m_{22} & b_{13} m_{21} & m_{24} \\[6pt]
			0 & 0 & m_{33} & m_{34} \\[6pt]
			0 & 0 & 0 & 0 
		\end{bmatrix}.
	\end{equation*}
This is symmetric if and only if $m_{14}=m_{24}=m_{34}=0$, $m_{21}=0$ and $m_{22}=m_{33}$ ($b_{13}=0$ leads to the same analysis as before with $b_{44}=b_{22}^{-2}$).

The algebra decomposes as $\mathfrak{g}= \langle X_1,X_2,X_3\rangle \oplus X_4\mathbb{R}$ where the first summand is the abelian nilradical $\mathfrak{n}$. Now note that
\begin{equation*}
	\ad_{X_4}|_{\mathfrak{n}}=\begin{bmatrix} 
		0 & b_{44} & 0 \\
		-b_{44} & 0 & -b_{13}b_{44} \\
		0 & 0 & 0 
	\end{bmatrix}
\end{equation*}
In view of Theorem \ref{teoB}, no solitons with $H\in \Lambda^3\mathfrak{n}^*$ exist, since $\ad_{X_4}^T \in \Der(\mathfrak{g})$ can only happen if $b_{13}=0$ and this reduces to the same analysis we did for the previous isometry class. In particular, since the Ricci operator is represented by the matrix
\begin{equation*}
	\begin{bmatrix}
		0 & 0 & -\frac{1}{2}b_{13}b_{44}^2 & 0 \\
		0 & \frac{1}{2}b_{13}^2 b_{44}^2 & 0 & 0 \\
		-\frac{1}{2}b_{13}b_{44}^2 & 0 & -\frac{1}{2}b_{13}^2 b_{44}^2 & 0 \\
		0 & 0 & 0 & -\frac{1}{2}b_{13}^2 b_{44}^2
	\end{bmatrix},
\end{equation*} 
like before we can prove $\lambda_4=0$

Reasoning as in the previous isometry class, we get
\begin{align*}
	\lambda &= -\frac{b_{13}^{2} b_{44}^{2}}{2} - \frac{\lambda_{3}^{2}}{2} - \frac{\lambda_{2}^{2}}{2} - \frac{\lambda_{1}^{2}}{2} \\
	m_{33} &= \frac{\lambda_{3}^{2}}{2}
\end{align*}
and since
\begin{equation*}
	2\lambda+2m_{33}=-b_{13}^2 b_{44}^2 - \lambda_2^2 - \lambda_1^2 \leq 0
\end{equation*}  and $\rho(M)\omega=-2m_{33}\omega$, we conclude $\lambda_1=\lambda_2=b_{13}=0$ by Proposition \ref{propLaplaciano}. Then, computing $\langle\ric^B(X_1),X_1\rangle$ yields $\lambda_3=0$. As a consequence, $H=0$.
\end{proof}

\subsection{Algebras with non-abelian nilradical}
Up to isomorphism, there are two solvable, non-nilpotent, unimodular Lie algebras of dimension 4 with non-abelian, $3$-dimensional nilradical $\mathfrak{h}_3$ (that is the unique nilpotent non-abelian $3$-dimensional Lie algebra) spanned by three basis elements $\{e_1,e_2,e_3\}$ such that $\mu(e_2,e_3)=e_1$. The action of a fourth basis element $e_4$ is given in table \ref{table2}.
\setlength{\tabcolsep}{20pt}
\begin{table}[h]
	\centering
	\caption{Algebras with non-abelian nilradical}
	\begin{tabular}{lcl}
		\toprule
		$\mathfrak{s}$ & \qquad & $\ad_{e_4}|_\mathfrak{n}$ \\
		\midrule
		$\mathfrak{d}_{4}$ & \qquad& $\begin{bmatrix} 0 & 0 & 0 \\ 0 & 1 & 0 \\ 0 & 0 & -1 \end{bmatrix}$ \\ \addlinespace[10pt]
		$\mathfrak{d}'_{4,0}$ & \qquad& $\begin{bmatrix} 0 & 0 & 0 \\ 0 & 0 & 1 \\ 0 & -1 & 0\end{bmatrix}$ \\ 
		\bottomrule
	\end{tabular}
	\label{table2}
\end{table}

	\begin{proposition}
	On the Lie algebra $\mathfrak{d}_{4}$ the unique generalized Ricci soliton with nonzero $H$, up to simultaneous positive scaling of the pair $(\langle \cdot,\cdot\rangle, H)$ and replacement of $H$ with $-H$, is $(\mathfrak{d}_{4},\langle\cdot,\cdot\rangle, H, D, \lambda )$, where:
	\begin{itemize}
		\item  the bracket relations are, with respect to a $\langle \cdot,\cdot \rangle$-orthonormal basis $X_1, \dots, X_4$ and dual basis $X^1, \dots, X^4$:
		\begin{equation*}
			\mu(X_2, X_3)= X_1, \qquad \mu(X_2, X_4)=-X_2, \qquad \mu(X_3, X_4)=X_3;
		\end{equation*} 
		\item $H= X^{123}$;
		\item 
		$D = \diag(2,1,1,0)$.
		\item $\lambda=-2$;
	\end{itemize}
	\end{proposition}
	\begin{proof}
		This is the Lie algebra defined by
		\begin{equation*}
			\mu(e_2,e_3)=e_1, \qquad
			\mu(e_4,e_2)=e_2, \quad
			\mu(e_4,e_3)=-e_3.
		\end{equation*}
		Thanks to \cite{VanThuong2017}, every metric on $\mathfrak{g}:=\mathfrak{d}_4$ is isometric to one defined by an orthonormal basis $X_1,\dots,X_4$, with dual basis $X^1,\dots, X^4$ and brackets
		\begin{align*}
			\mu(X_2,X_3) &= \frac{1}{b_{11}}X_1,\\
			\mu(X_2,X_4) &= -b_{44}X_2+\frac{b_{44}b_{12}}{b_{11}}X_1,\\
			\mu(X_3,X_4) &= \frac{b_{44}}{b_{11}}\left(2b_{12}b_{23}-b_{13}\right)X_1-2b_{23}b_{44}X_2+b_{44}X_3.
		\end{align*}
		for constants $b_{11},b_{44}>0, b_{12},b_{23}\geq 0, b_{13} \in \mathbb{R}$. For these brackets, the Ricci operator is given by:
	\begin{align*}
		\ric_1^1 &= \frac{1 + b_{44}^2 b_{12}^2 + b_{44}^2 \left(2 b_{12} b_{23} - b_{13}\right)^2}{2 b_{11}^2}, \\
		\ric_1^2 &= -\frac{b_{44}^2 \left(4 b_{12} b_{23}^2 - 2 b_{13} b_{23} + b_{12}\right)}{2 b_{11}}, \\
		\ric_1^3  &= \frac{b_{44}^2 \left(2 b_{12} b_{23} - b_{13}\right)}{2 b_{11}}, \\
		\ric_1^4  &= 0, \\
		\ric_2^2 &= \frac{-1 - b_{44}^2 b_{12}^2 + 4 b_{11}^2 b_{23}^2 b_{44}^2}{2 b_{11}^2}, \\
		\ric_2^3  &= \frac{-b_{12} b_{44}^2 \left(2 b_{12} b_{23} - b_{13}\right) - 4 b_{11}^2 b_{44}^2 b_{23}}{2 b_{11}^2}, \\
		\ric_2^4 &= \frac{b_{44} \left(2 b_{12} b_{23} - b_{13}\right)}{2 b_{11}^2}, \\
		\ric_3^3 &= \frac{-1 - b_{44}^2 \left(2 b_{12} b_{23} - b_{13}\right)^2 - 4 b_{11}^2 b_{23}^2 b_{44}^2}{2 b_{11}^2}, \\
		\ric_3^4  &= -\frac{b_{44} b_{12}}{2 b_{11}^2}, \\
		\ric_4^4 &= \frac{-b_{44}^2 b_{12}^2 - b_{44}^2 \left(2 b_{12} b_{23} - b_{13}\right)^2 - 4 b_{11}^2 b_{44}^2 - 4 b_{11}^2 b_{23}^2 b_{44}^2}{2 b_{11}^2}.
	\end{align*}
		A generic derivation of the Lie algebra is of the form:
		\begin{equation*}
			D= \begin{bmatrix}
				m_{11} & m_{12} & m_{13} & m_{14} \\[6pt]
				0 & m_{22} & -(m_{11}-2 m_{22}) b_{23} & b_{44} [b_{11} (2 b_{23} m_{12}-m_{13})-b_{12} b_{23} m_{11}+b_{13} m_{22}] \\[6pt]
				0 & 0 & -m_{22}+m_{11} & -b_{44} [(-m_{11}+m_{22}) b_{12}+b_{11} m_{12}] \\[6pt]
				0 & 0 & 0 & 0 
			\end{bmatrix}
		\end{equation*}
	and we get a non-trivial symmetric derivation if $m_{12}=m_{13}=m_{14}=0$ and one of the following holds:
	\begin{itemize}
		\item $b_{23}=b_{13}=b_{12}=0$;
		\item $b_{12}=b_{13}=0,$ $m_{11}=2m_{22}$;
		\item $m_{11}=m_{22}$, $b_{23}=b_{13}=0$.
	\end{itemize} 
	
		 Let $H=\lambda_4 X^{123}+\lambda_3 X^{124}+\lambda_2 X^{134}+\lambda_1 X^{234}$ be a generic $3$-form on $\mathfrak{g}$. Then, the endomorphism $H^2$ has the same form as in \eqref{H_squared}. Also, using the splitting in Proposition \ref{propLaplaciano} we write $H=X^4\wedge \omega+\eta$, where $\omega=\lambda_3 X^{12}+\lambda_2 X^{13}+\lambda_1 X^{23}$ and $\eta=\lambda_4 X^{123}$. Then,  $\ric-\lambda\Id-\frac{1}{4}H^2\in \Der(\mathfrak{g})$ only if the constant $\lambda$ is:
		\begin{equation*}
			\begin{split}
				\lambda = -\frac{1}{2 b_{11}^{2}} \Big[ & 4 b_{44}^{2} \left( b_{11}^{2} b_{23}^{2} + b_{12}^{2} b_{23}^{2} - b_{12} b_{13} b_{23} + b_{11}^{2} \right)+ \\
				& + b_{11}^{2} \left( \lambda_3^{2} + \lambda_2^{2} + \lambda_1^{2} \right) + b_{44}^{2} \left( b_{12}^{2} + b_{13}^{2} \right) \Big].
			\end{split}
		\end{equation*}
		
		Let us consider the case $b_{23}=b_{13}=b_{12}=0$. Assume that $\lambda_1=\lambda_2=\lambda_3=0$. Then, computing $\langle \ric^B(X_1),X_1\rangle$, $\langle \ric^B(X_2),X_2\rangle$ and using the first equation in \eqref{algebraic soliton} yields:
		\begin{align*}
			\lambda &= -2 b_{44}^2, \\
			m_{11} &= \frac{1 - \lambda_4^2 b_{11}^2 + 4 b_{11}^2 b_{44}^2}{2 b_{11}^2}, \\
			m_{22} &= \frac{-1 - \lambda_4^2 b_{11}^2 + 4 b_{11}^2 b_{44}^2}{2 b_{11}^2}.
		\end{align*}
		At this point, the GRS equation $\Delta H=2\lambda H-\rho(D)H$ implies that $\lambda_4^2=\frac{1}{b_{11}^2}$ and from the computation of $\langle \ric^B(X_3),X_3\rangle$ we get $b_{44}^2b_{11}^2=1$. This provides a nontrivial soliton as in the statement.
		
		Now we assume that at least one among $\lambda_1,\lambda_2,\lambda_3$ is nonzero. Then, $\lambda_4=0$ and again, if we compute $\langle \ric^B(X_1),X_1\rangle$, $\langle \ric^B(X_2),X_2\rangle$ and use the first equation in \eqref{algebraic soliton} we obtain:
		\begin{align*}
			\lambda &= -2 b_{44}^2 - \frac{1}{2}\left(\lambda_1^2 + \lambda_2^2 + \lambda_3^2\right), \\[6pt]
			m_{11} &= \frac{1 + 4 b_{11}^2 b_{44}^2 + b_{11}^2 \lambda_1^2}{2 b_{11}^2}, \\[6pt]
			m_{22} &= \frac{-1 + 4 b_{11}^2 b_{44}^2 + b_{11}^2 \lambda_2^2}{2 b_{11}^2}.
		\end{align*}
	At this point, from $\Delta H=2\lambda H-\rho(D)H$, we have the following equations:
	\begin{equation*}
		\begin{cases}
			\lambda_1 \left(4b_{11}^2 b_{44}^2 + 2b_{11}^2 \lambda_3^2 + 2b_{11}^2 \lambda_2^2 + b_{11}^2 \lambda_1^2 + 1\right)=0 \\
			2\lambda_3^3 + \left(2b_{44}^2 + \lambda_2^2 + \lambda_1^2\right)\lambda_3=0\\
			\lambda_2 \left(6b_{11}^2 b_{44}^2 + 2b_{11}^2 \lambda_3^2 + 3b_{11}^2 \lambda_2^2 - 3\right)=0.
		\end{cases}
	\end{equation*}
	As a consequence, $\lambda_1=\lambda_3=0$ and $2b_{44}^2b_{11}^2=1-\lambda_2^2b_{11}^2$.
	 But this is not compatible with the fact that $\ric-\frac{1}{4}H^2-\lambda \Id \in \Der(\mathfrak{g})$. Indeed, $\langle \ric^B(X_3),X_3\rangle$ forces $2b_{11}^2b_{44}^2 =- b_{11}^2\lambda_2^2 + 3$. As a result, there are no GRS in the case $b_{23}=b_{13}=b_{12}=0$.
		
	The remaining cases that yield symmetric derivations are treated similarly, noticing that we may reduce to the case we just analyzed.	
	\end{proof}

	\begin{proposition}
		\label{class2}
		On the Lie algebra $\mathfrak{d}'_{4,0}$ there do not exist any algebraic generalized Ricci solitons with $H\neq 0$.
	\end{proposition}
	\begin{proof}
		This is the Lie algebra with brackets defined by
		\begin{equation*}
			\mu(e_2,e_3) = e_1, \qquad
			\mu(e_4,e_2) = -e_3, \qquad
			\mu(e_4,e_3) = e_2.
		\end{equation*}
		Thanks to \cite{VanThuong2017} every metric on $\mathfrak{g}=\mathfrak{d}'_{4,0}$ is isometric to one of the following two isometry classes, both defined for an orthonormal basis $X_1,\dots X_4$ with dual basis $X^1,\dots,X^4$. The first one is obtained with brackets relations
		\begin{equation*}
			\mu(X_2,X_3) = \frac{b_{33}}{b_{11}}X_1,\qquad
			\mu(X_2,X_4) = -\frac{b_{13}b_{44}}{b_{11}b_{33}}X_1+\frac{b_{44}}{b_{33}}X_3,\qquad
			\mu(X_3,X_4) = \frac{b_{12}b_{33}b_{44}}{b_{11}}X_1-b_{33}b_{44}X_2.
		\end{equation*}
		and $b_{11},b_{44}>0, b_{12},b_{13}\geq0, 0<b_{33}<1$.
		
		The second one is obtained with brackets
		\begin{equation*}
			\mu(X_2,X_3) = \frac{1}{b_{11}}X_1, \qquad
			\mu(X_2,X_4) = b_{44}X_3, \qquad
			\mu(X_3,X_4) = \frac{b_{12}b_{44}}{b_{11}}X_1-b_{44}X_2.
		\end{equation*}
		and $b_{11},b_{44}>0,b_{12}\geq0$.
		
		Let us consider the first isometry class. For those brackets, the Ricci operator is represented by the following matrix:
		\begin{equation*}
			\ric=\frac{1}{2 b_{11}}
			\resizebox{\textwidth}{!}{$
			\begin{bmatrix}
				\frac{b_{12}^{2} b_{33}^{4} b_{44}^{2}+b_{13}^{2} b_{44}^{2}+b_{33}^{4}}{b_{11} b_{33}^{2}} & -b_{44}^{2} b_{12} b_{33}^{2} & -\frac{b_{13} b_{44}^{2}}{b_{33}^{2}} & 0 \\
				-b_{44}^{2} b_{12} b_{33}^{2} & \frac{b_{44}^{2} \left(b_{33}^{4}-1\right) b_{11}^{2}-b_{13}^{2} b_{44}^{2}-b_{33}^{4}}{b_{11} b_{33}^{2}} & \frac{b_{13} b_{44}^{2} b_{12}}{b_{11}} & \frac{b_{33}^{2} b_{44} b_{12}}{b_{11}} \\
				-\frac{b_{13} b_{44}^{2}}{b_{33}^{2}} & \frac{b_{13} b_{44}^{2} b_{12}}{b_{11}} & \frac{\left(-b_{44}^{2} b_{11}^{2}-b_{12}^{2} b_{44}^{2}-1\right) b_{33}^{4}+b_{44}^{2} b_{11}^{2}}{b_{11} b_{33}^{2}} & \frac{b_{13} b_{44}}{b_{11}} \\
				0 & \frac{b_{33}^{2} b_{44} b_{12}}{b_{11}} & \frac{b_{13} b_{44}}{b_{11}} & -\frac{\left(\left(b_{33}^{4}-2 b_{33}^{2}+1\right) b_{11}^{2}+b_{12}^{2} b_{33}^{4}+b_{13}^{2}\right) b_{44}^{2}}{b_{11} b_{33}^{2}}
			\end{bmatrix}$}.
		\end{equation*}
		
	A generic derivation is
		\begin{equation*}
			D=\begin{bmatrix}
				2 m_{33} & m_{12} & m_{13} & m_{14} \\
				0 & m_{33} & -m_{32} b_{33}^2 & (b_{11}m_{12} - b_{12}m_{33} + b_{13}m_{32})b_{44}\\
				0 & m_{32} & m_{33} & \frac{b_{44}(-b_{12}b_{33}^2m_{32} + b_{11}m_{13} - b_{13}m_{33})}{b_{33}^2} \\
				0 & 0 & 0 & 0
			\end{bmatrix}
		\end{equation*}
		so that $D$ is a nontrivial symmetric derivation if and only if $m_{12}=m_{13}=m_{14}=m_{32}=b_{12}=b_{13}=0$.
		
		Let $H=\lambda_4 X^{123}+\lambda_3 X^{124}+\lambda_2 X^{134}+\lambda_1 X^{234}$ be a generic $3$-form on $\mathfrak{g}$. Then, the endomorphism $H^2$ has the same form as in \eqref{H_squared}. Also, using the splitting in Proposition \ref{propLaplaciano} we write $H=X^4\wedge \omega+\eta$, where $\omega=\lambda_3 X^{12}+\lambda_2 X^{13}+\lambda_1 X^{23}$ and $\eta=\lambda_4 X^{123}$.
		
		As usual, assume $\lambda_1=\lambda_2=\lambda_3=0$. Then, $\ric^B-\lambda\Id \in \Der(\mathfrak{g})$ only if
		\begin{align*}
			\lambda &= -\frac{\left(b_{33}-1\right)^2\left(b_{33} +1\right)^2b_{44}^2}{2 b_{33}^{2}}, \\
			m_{33} &= \frac{b_{11}^{2} b_{33}^{4} b_{44}^{2} - 2 b_{11}^{2} b_{33}^{2} b_{44}^{2} - b_{11}^{2} b_{33}^{2} \lambda_{4}^{2} + b_{11}^{2} b_{44}^{2} + b_{33}^{4}}{4 b_{11}^{2} b_{33}^{2}}.
		\end{align*}
		The second equation in \eqref{algebraic soliton} simplifies to
		\begin{equation*}
		b_{11}^2 \lambda_4^3 - b_{33}^2\lambda_4=0
		\end{equation*} 
		meaning that $b_{33}=\pm b_{11}\lambda_4$.
		At this point, the computations of $\langle\ric^B(X_2),X_2\rangle$ and $\langle\ric^B(X_3),X_3\rangle$ yield
	\begin{align*}
		-b_{44}^2 + 3b_{11}^4 b_{44}^2 \lambda_4^4 + \left(-4\lambda_4^4 - 2b_{44}\lambda_4^2\right) b_{11}^2& =0\\
		 3b_{44}^2 - b_{11}^4 b_{44}^2 \lambda_4^4 + \left(-4\lambda_4^4 - 2b_{44}\lambda_4^2\right) b_{11}^2& = 0.
	\end{align*}
		We cannot find a common solution to both these equations, with both $b_{44},b_{11}>0$. This can be seen, for instance, by subtracting the two equations above.
		
		The case where $\lambda_1,\lambda_2,\lambda_3$ are not all zero follows by reasoning in analogy with the previous isometry class. In order that $\ric^B-\lambda\Id \in \Der(\mathfrak{g})$ it is necessary that $\lambda_4=0$ and exactly two among $\lambda_1,\lambda_2,\lambda_3$ vanish. Every choice leads to a contradiction with the equations in \eqref{algebraic soliton}, indeed:
		\begin{itemize}
			\item if $\lambda_1=\lambda_3=0$, then the computations of $\langle\ric^B(X_i),X_i\rangle$ with $i=4,3,2$ yield:
			\begin{align*}
				\lambda &= -\frac{\left(b_{33}-1\right)^2\left(b_{33} +1\right)^2b_{44}^2}{2 b_{33}^{2}}-\frac{\lambda_2^2}{2}, \\
				m_{33} &= -\frac{  2 b_{11}^{2} b_{33}^{2} b_{44}^{2}-2b_{44}^2b_{11}^{2} + b_{33}^{4}}{2 b_{11}^{2} b_{33}^{2}},\\
				\lambda_2&= \pm \frac{b_{44}\sqrt{2-2b_{33}^4}}{b_{33}}.
			\end{align*}
			Note that $2-2b_{33}^4>0$ since $b_{33}\in (0,1)$, so the square root makes sense. But then, the second equation in \eqref{algebraic soliton} forces $2b_{11}^2b_{44}^2 + 3b_{33}^2=0$ which is not possible.
			
			\item if $\lambda_1=\lambda_2=0$ we reason exactly as before, finding $\lambda, m_{33}$ from $\langle\ric^B(X_i),X_i\rangle$ with $i=4,3$. But now, for $i=2$ we get that $2b_{33}^4b_{44}^2 - b_{33}^2\lambda_3^2 - 2b_{44}^2=0$ and we cannot solve this equation for $\lambda_3$, since $b_{33}\in (0,1)$.
			
			\item if $\lambda_2=\lambda_3=0$ we follow the same reasoning of the previous two cases, computing finding $\lambda, m_{33}$ from $\langle\ric^B(X_i),X_i\rangle$ with $i=4,3$. Then, $i=2$ leads to the equation $b_{44}^2(b_{33}^4-1)=0$ which does not have any real solution.
		\end{itemize}
		In any case, we do not find any algebraic GRS in this isometry class.
		
		Let us consider now the second isometry class. In this case the Ricci operator is represented by the following matrix, with respect to the basis $\{X_1,\dots,X_4\}$:
		\begin{equation*}
			\ric=\frac{1}{2 b_{11}^{2}}
			\begin{bmatrix}
				b_{44}^{2} b_{12}^{2}+1 & -b_{11} b_{44}^{2} b_{12} & 0 & 0 \\[1.5ex]
				-b_{11} b_{44}^{2} b_{12} & -1 & 0 & b_{44} b_{12} \\[1.5ex]
				0 & 0 & -b_{44}^{2} b_{12}^{2}-1 & 0 \\[1.5ex]
				0 & b_{44} b_{12} & 0 & -b_{44}^{2} b_{12}^{2}
			\end{bmatrix}
		\end{equation*}
		and a generic derivation has the form
		\begin{equation*}
			D= \begin{bmatrix}
				2 m_{33} & m_{12} & m_{13} & m_{14} \\[6pt]
				0 & m_{33} & -m_{32} & b_{44}(b_{11}m_{12}- b_{12}m_{33}) \\[6pt]
				0 & m_{32} & m_{33} & b_{44}(b_{11}m_{13}- b_{12}m_{32})\\[6pt]
				0 & 0 & 0 & 0 
			\end{bmatrix}
		\end{equation*}
	so that we only obtain nontrivial symmetric derivations for $m_{12}=m_{13}=m_{14}=m_{32}=0$ and $b_{12}=0$.
	
	Assume that $\lambda_1=\lambda_2=\lambda_3=0$. Then, in order that $\ric^B-\lambda \Id \in \Der(\mathfrak{g})$ we need that
	\begin{align*}
	  \lambda&=0,\\ m_{33}&=-\frac{1}{2b_{11}^2}-\frac{\lambda_4^2}{2}.
	\end{align*}
	 But now, computing $\langle \ric^B(X_1), X_1\rangle$ and using that $\ric^B-\lambda\Id \in \Der(\mathfrak{g})$ we obtain that
	 \begin{equation*}
	 	3+b_{11}^2\lambda_4^2=0
	 \end{equation*}
	 which is impossible.
		
		Therefore, assume that at least one among $\lambda_1,\lambda_2,\lambda_3$ is nonzero. This forces $\lambda_4=0$ and 
		\begin{align*}
			\lambda &= - \frac{1}{2}\left(\lambda_1^2 + \lambda_2^2 + \lambda_3^2\right), \\[6pt]
		m_{33} &= -\frac{1}{2b_{11}^2}+\frac{\lambda_3^2}{2}.
		\end{align*}
		The second equation in \eqref{algebraic soliton} now implies $\lambda_1=0$. Moreover, by computing $\langle \ric^B(X_1), X_1\rangle$ and using that $\ric^B-\lambda\Id \in \Der(\mathfrak{g})$ we deduce that $\lambda_3=\pm \sqrt{\frac{3}{2b_{11}^2}}$ and this forces $\lambda_2=0$ from $\langle\ric^B(X_2),X_3 \rangle$. Now, by computing $\langle \ric^B(X_2), X_2\rangle$ we see that it is impossible to have $\ric^B-\lambda\Id \in \Der(\mathfrak{g})$, since we would need $-\frac{3}{4b_{11}^2}=0$.
		
		As a result no algebraic GRS arise from this isometry class.
	\end{proof}

The examples arising from this classification are summarized in Table \ref{table3} below.

\begin{table}[ht]
	\centering
	\small
	\setlength{\tabcolsep}{4pt}
	\renewcommand{\arraystretch}{1.25}
	\caption{Four-dimensional solvable non-nilpotent algebraic generalized Ricci solitons with $H\neq 0$
		on unimodular Lie algebras.}
	\label{table3}
	\begin{tabular}{c|c|c|c|c}
		\hline
		Lie algebra & Existence & $H$ & $D$ & $\lambda$ \\
		\hline
		$\mathfrak{r}_{4,-\frac12,-\frac12}$
		&
		YES
		&
		$ 2X^{123}$
		&
		$4\,\mathrm{diag}(1,1,1,0)$
		&
		$-6$
		\\[3pt]
		
		$\mathfrak{e}(1,1)\oplus\mathbb R$
		&
		YES
		&
		$\displaystyle \frac{2}{\sqrt3}X^{123}$
		&
		$\displaystyle \frac43\,\mathrm{diag}(1,1,1,0)$
		&
		$-2$
		\\[6pt]
		
		$\mathfrak{r}'_{4,-2\mu,\mu}$, $\mu>0$
		&
        YES
		&
		$2\mu X^{123}$
		&
		$\displaystyle 4\mu^2\,\mathrm{diag}(1,1,1,0)$
		&
		$-6\mu^2$
		\\[6pt]
		
		$\mathfrak{r}_{4,\mu,-1-\mu}$, $\mu>1$
		&
		YES
		&
		$\displaystyle
		\frac23
		\sqrt{3(\mu^2+\mu+1)}\,X^{123}$
		&
		$\displaystyle
		\frac43(\mu^2+\mu+1)\,
		\mathrm{diag}(1,1,1,0)$
		&
		$\displaystyle -2(\mu^2+\mu+1)$
		\\[8pt]
		
		$\mathfrak{r}_{4,-\frac12}$
		&
		NO
		&
		---
		&
		---
		&
		---
		\\[3pt]
		
		$\mathfrak{e}(2)\oplus\mathbb R$
		&
		NO
		&
		---
		&
		---
		&
		---
		\\[3pt]
		
		$\mathfrak{d}_4$
		&
		YES
		&
		$ X^{123}$
		&
		$\mathrm{diag}(2,1,1,0)$
		&
		$-2$
		\\[3pt]
		
		$\mathfrak{d}'_{4,0}$
		&
		NO
		&
		---
		&
		---
		&
		---
		\\
		\hline
	\end{tabular}
\end{table}
	\printbibliography
	
\end{document}